\pdfoutput=1
\RequirePackage[l2tabu, orthodox]{nag}

\documentclass[reqno,tbtags]{amsart}

\usepackage{lmodern}
\usepackage[T1]{fontenc}
\usepackage[utf8]{inputenc}
\usepackage[english]{babel}
\usepackage{microtype} 

\usepackage{datetime}
\usepackage[english]{isodate}

\usepackage[colorinlistoftodos]{todonotes}

\usepackage{amsmath,amssymb,amsthm,mathrsfs,latexsym,mathtools,mathdots,xparse,
booktabs,array,enumerate,tikz,hyperref}
\usepackage[export]{adjustbox} 
\usepackage[centertableaux]{ytableau}

\usetikzlibrary{arrows,positioning}
\usepackage[capitalize,noabbrev]{cleveref}

\usepackage{graphicx}
\DeclareGraphicsExtensions{.pdf}

\newtheorem{theorem}{Theorem}
\newtheorem{proposition}[theorem]{Proposition}
\newtheorem{lemma}[theorem]{Lemma}
\newtheorem{corollary}[theorem]{Corollary}
\newtheorem{conjecture}[theorem]{Conjecture}

\theoremstyle{definition}
\newtheorem{definition}[theorem]{Definition}
\newtheorem{example}[theorem]{Example}
\newtheorem{remark}[theorem]{Remark}

\newtheorem{problem}[theorem]{Problem}

\newtheorem{bijection}{Bijection} 

\NewDocumentCommand\qbinom{O{q} m m o}{
\ensuremath{\IfNoValueTF{#4}{
{\genfrac{[}{]}{0pt}{}{#2}{#3}_{#1}}
}{
{\genfrac{[}{]}{0pt}{}{#2}{#3}_{#1}^{#4}}
}}}

\newmuskip\pFqmuskip
\newcommand*\pFq[6][8]{%
  \begingroup 
  \pFqmuskip=#1mu\relax
  \mathcode`\,=\string"8000
  \begingroup\lccode`\~=`\,
  \lowercase{\endgroup\let~}\pFqcomma
  {}_{#2}\phi_{#3}{\left[\genfrac..{0pt}{}{#4}{#5};#6\right]}%
  \endgroup
}
\newcommand{\pFqcomma}{\mskip\pFqmuskip}

\newcommand{\setN}{\mathbb{N}}
\newcommand{\setZ}{\mathbb{Z}}

\newcommand{\schurS}{\mathrm{s}}

\NewDocumentCommand\Cat{ o m }{%
\ensuremath{\IfNoValueTF{#1}{\mathrm{Cat}({#2})}{\mathbf{Cat}({#2};{#1})}}}

\NewDocumentCommand\Nar{ o m m }{%
\ensuremath{\IfNoValueTF{#1}{\mathrm{Nar}({#2,#3})}{\mathbf{Nar}({#2,#3};{#1})}}}

\NewDocumentCommand\CatB{ o m }{%
\ensuremath{\IfNoValueTF{#1}{\mathrm{Cat}^B({#2})}{\mathbf{Cat}^B({#2};{#1})}}}

\NewDocumentCommand\NarB{ o m m }{%
\ensuremath{\IfNoValueTF{#1}{\mathrm{Nar}^B({#2,#3})}{\mathbf{Nar}^B({#2,#3};{#1})}}}

\NewDocumentCommand\qSYT{ o m m }{%
\ensuremath{\IfNoValueTF{#1}{\mathrm{Syt}({#2,#3})}{\mathbf{Syt}({#2,#3};{#1})}}}

\NewDocumentCommand\qNCC{ o m m m }{%
\ensuremath{\IfNoValueTF{#1}{\mathrm{Ncc}({#2,#3,#4})}{\mathbf{Ncc}({#2,#3,#4};{#1})}}}

\NewDocumentCommand\qBW{ o m m }{%
\ensuremath{\IfNoValueTF{#1}{\mathrm{Bw}({#2,#3})}{\mathbf{Bw}({#2,#3};{#1})}}}

\NewDocumentCommand\qNCCB{ o m m m }{%
\ensuremath{\IfNoValueTF{#1}{\mathrm{Ncc}^B({#2,#3,#4})}{\mathbf{Ncc}^B({#2,#3,#4};{#1})}}}

\NewDocumentCommand\Tri{ o m m }{%
\ensuremath{\IfNoValueTF{#1}{\mathrm{Tri}({#2,#3})}{\mathbf{Tri}({#2,#3};{#1})}}}

\newcommand*{\oeis}[1]{\href{http://oeis.org/#1}{#1}}
\newcommand*{\defin}[1]{\textcolor{blue}{\emph{#1}}}
\newcommand*{\mdefin}[1]{\textcolor{blue}{#1}}

\newcommand{\BW}{\mathrm{BW}}  
\newcommand{\SSYT}{\mathrm{SSYT}}
\newcommand{\SYT}{\mathrm{SYT}}

\newcommand{\NCP}{\mathrm{NCP}}
\newcommand{\NCPB}{\mathrm{NCP}^B}
\newcommand{\NCM}{\mathrm{NCM}}

\newcommand{\NCC}{\mathrm{NCC}} 

\newcommand{\DYCK}{\mathrm{DYCK}}
\newcommand{\PATHS}{\mathrm{PATH}} 
\newcommand{\TRI}{\mathrm{TRI}} 

\newcommand{\DES}{\mathrm{Des}}
\newcommand{\CDES}{\mathrm{cDes}}

\newcommand{\sfw}{\mathsf{w}}
\newcommand{\sfb}{\mathsf{b}}
\newcommand{\sfc}{\mathsf{c}}

\newcommand{\thsup}{\textnormal{th}}

\DeclareMathOperator{\peaks}{peaks}
\DeclareMathOperator{\valleys}{valleys}

\DeclareMathOperator{\blocks}{blocks}

\DeclareMathOperator{\des}{des}

\DeclareMathOperator{\even}{even}
\DeclareMathOperator{\order}{o}

\DeclareMathOperator{\inv}{inv}
\DeclareMathOperator{\maj}{maj}
\DeclareMathOperator{\pmaj}{pmaj}

\DeclareMathOperator{\cdes}{cdes}
\DeclareMathOperator{\sh}{sh}
\DeclareMathOperator{\sml}{short}
\DeclareMathOperator{\st}{st}
\DeclareMathOperator{\dep}{depth}
\DeclareMathOperator{\ear}{ear}

\DeclareMathOperator{\rot}{rot} 
\newcommand{\rotB}{\mathrm{rotB}} 
\DeclareMathOperator{\twist}{twist} 
\DeclareMathOperator{\shift}{shift} 
\DeclareMathOperator{\rowmotion}{\rho} 
\DeclareMathOperator{\flip}{\gamma} 
\DeclareMathOperator{\krew}{krew} 

\newcommand{\laser}{\mathtt{DYCKtoNCC}}

\newcommand{\prom}{\partial}

\newcommand{\kprom}{\hat{\partial}}

\newcommand{\SYTtoDYCK}{\mathtt{SYTtoDYCK}}

\newcommand{\SYTtoNCM}{\mathtt{SYTtoNCM}}

\newcommand{\NCMtoDYCK}{\mathtt{NCMtoDYCK}}

\newcommand{\NCPtoNCM}{\mathtt{NCPtoNCM}}

\newcommand{\BWtoNCM}{\mathtt{BWtoNCM}}

\newcommand{\NCPtoDYCK}{\mathtt{NCPtoDYCK}}

\newcommand{\matching}[3]{%
	\begin{tikzpicture}[anchor=base, baseline]
	\draw (0,0) circle (0.8);
	\foreach \x in {1,...,#1} {
		\node[shape=circle,fill=black, scale=0.5,label={{((-\x+1)*360/#1)+90}:$\scriptstyle{\x}$}] (n\x) at ({((-\x+1)*360/#1)+90}:0.8) {};
	};
	\foreach \x/\y in {#2} {

		\draw[thick] (n\x) -- (n\y);
	};
	\foreach \x in {#3} {
		\draw (n\x) circle (0.15);
	};
	\end{tikzpicture}%
}

\newcommand{\DP}{\mathrm{DP}}
\newcommand{\NEpath}[4]{
	\fill[white!25]  (#1) rectangle +(#2,#3);
	\fill[fill=white]
	(#1)
	\foreach \dir in {#4}{
		\ifnum\dir=0
		-- ++(1,0)
		\else
		-- ++(0,1)
		\fi
	} |- (#1);
	\draw[help lines] (#1) grid +(#2,#3);
	\draw[dashed] (#1) -- +(#3,#3);
	\coordinate (prev) at (#1);
	\foreach \dir in {#4}{
		\ifnum\dir=0
		\coordinate (dep) at (1,0);
		\else
		\coordinate (dep) at (0,1);
		\fi
		\draw[line width=2pt,-stealth] (prev) -- ++(dep) coordinate (prev);
	};
}

\newcounter{DyckHsteps}
\tikzset{count list/.code 2 args={\foreach \XX [count=\YY] in {#1}
		{\xdef#2{\YY}}},Dyck arrow/.style={ultra thick,-stealth},
	laser/.style={draw=blue},
	Dyck path/.style={count list={#1}{\DyckSteps},
		/utils/exec=\setcounter{DyckHsteps}{0},insert path={%
			foreach \XX [count=\YY,remember=\YY as \LastY (initially 0)]in {#1}
			{\ifnum\XX=0
				edge[Dyck arrow] ++(1,0) ++(1,0) coordinate(Dyck-\YY)
				\ifnum\YY<\DyckSteps
				(Dyck-\LastY) -- (Dyck-\YY) node[midway,above]{\stepcounter{DyckHsteps}\number\value{DyckHsteps}}
				\fi
				\else
				edge[Dyck arrow] ++(0,1) ++(0,1) coordinate(Dyck-\YY)
				\fi
				\pgfextra{\pgfmathtruncatemacro{\vtest}{0}\pgfmathtruncatemacro{\ftest}{0}\pgfmathtruncatemacro{\htest}{0}\pgfmathtruncatemacro{\itest}{1}
					\pgfmathtruncatemacro{\RestSteps}{\DyckSteps-\YY}
					\ifnum\YY>1
					\ifnum\RestSteps>1
					\pgfmathtruncatemacro{\ftest}{{#1}[\YY+1]+{#1}[\YY]*10} 
					\pgfmathtruncatemacro{\vtest}{{#1}[\YY-1]+10*{#1}[\YY]} 
					\fi
					\ifnum\RestSteps>3
					\pgfmathtruncatemacro{\htest}{pow(-1,{#1}[\YY+3])+pow(-1,{#1}[\YY+2])
						+pow(-1,{#1}[\YY+1])+pow(-1,{#1}[\YY])+ifthenelse({#1}[\YY-1]==1,11,0))}
					\fi
					\ifnum\RestSteps>5
					\pgfmathtruncatemacro{\itest}{pow(-1,{#1}[\YY+5])+
						pow(-1,{#1}[\YY+4])+pow(-1,{#1}[\YY+3])+pow(-1,{#1}[\YY+2])
						+pow(-1,{#1}[\YY+1])+pow(-1,{#1}[\YY])+ifthenelse({#1}[\YY-1]==1,11,0)
						+ifthenelse({#1}[\YY-2]==1,11,0)}
					\fi
					\fi
				}
				\ifnum\vtest=10
				\ifnum\itest=0
				(Dyck-\YY) edge[laser] ++(3,3) (Dyck-\YY)
				\fi
				\ifnum\htest=1100
				(Dyck-\YY) edge[laser] ++(-2,-2) (Dyck-\YY)
				\fi
				\ifnum\ftest=10
				(Dyck-\YY) edge[laser] ++(1,1) (Dyck-\YY)
				\fi
				\fi
}}}}

\title[Refinements of cyclic sieving]{Refined Catalan and Narayana cyclic sieving}
\date{\today, \currenttime}

\author{Per Alexandersson}
\email{per.w.alexandersson@gmail.com}
\address{Dept.~of Mathematics,
Stockholm University,
SE-10691, Stockholm, Sweden}

\author{Svante Linusson}
\email{linusson@math.kth.se}
\address{Dept.~of Mathematics, KTH Royal Institute of Technology, SE-100 44 Stockholm, Sweden}

\author{Samu Potka}
\email{potka@kth.se}
\address{Dept.~of Mathematics, KTH Royal Institute of Technology, SE-100 44 Stockholm, Sweden}

\author{Joakim Uhlin}
\email{joakim\_uhlin@hotmail.com}
\address{Dept.~of Mathematics,
Stockholm University,
SE-10691, Stockholm, Sweden}


\begin{document}

\begin{abstract}
We prove several new instances of the cyclic
sieving phenomenon (CSP) on Catalan objects of type $A$ and type $B$.
Moreover, we refine many of the known instances of the CSP on Catalan objects.
For example, we consider triangulations refined by the number of ``ears'',
non-crossing matchings with a fixed number of short edges,
and non-crossing configurations with a fixed number of loops and edges.
\end{abstract}

\maketitle

\setcounter{tocdepth}{1}
\tableofcontents

\section{Introduction}
The original inspiration for this paper is a natural interpolation between type $A$ and type $B$
Catalan numbers. For $n\geq 0$ consider the expression
\begin{equation}\label{eq:introInterpolation}
\binom{2n}{n} - \binom{2n}{n-s-1}.
\end{equation}
For $s=0$, we recover the $n^\thsup$ Catalan number and
for $s=1$, we recover the $(n+1)^\thsup$ Catalan number. When $s=n$,
we obtain the central binomial coefficient $\binom{2n}{n}$, which is
known as the $n^\thsup$ type~$B$~Catalan number, see~\cite{Armstrong2009}.
There are several combinatorial families of objects which are counted by the expression in
\eqref{eq:introInterpolation}, certain standard Young tableaux and lattice paths to name a few.
The expression in \eqref{eq:introInterpolation}
has the $q$-analog given by the difference of $q$-binomials
\begin{equation}\label{eq:introqInterpolation}
\qbinom{2n}{n} - q^{s+1}\qbinom{2n}{n-s-1}.
\end{equation}
For $s \in \{0,1,n\}$, the polynomials in \eqref{eq:introqInterpolation}
appear in instances of the \emph{cyclic sieving phenomenon}. Furthermore, it follows from
 \cite[Theorem 46]{AlexanderssonPfannererRubeyUhlin2020x} that there exist group actions
such that the polynomials in \eqref{eq:introqInterpolation} exhibit cyclic sieving for
all $s \in \{0,1,\dotsc,n\}$.

\begin{definition}[Cyclic sieving, \cite{ReinerStantonWhite2004}]
	Let $X$ be a set and $C_n$ be the cyclic group of order $n$
	acting on $X$. Let $f(q)\in \setN[q]$.
	We say that the \defin{triple} $(X,C_n,f(q))$ \defin{exhibits the cyclic sieving phenomenon (CSP)}
	if for all $d \in \setZ$,
	\begin{align}\label{eq:cspDef}
	|\{ x\in X : g^d \cdot x = x \}| = f(\xi^d)
	\end{align}
	where $\xi$ is a primitive $n^\thsup$ root of unity.
\end{definition}
Note that it follows immediately from the definition that $|X| = f(1)$.
In the study of cyclic sieving, it is mainly the case that the $C_n$-action and the
polynomial $f(q)$ are natural in some sense.
The group action could be some form of rotation or cyclic shift of
the elements of $X$. The polynomial usually has a closed form and is
also typically the generating polynomial for some combinatorial statistic defined on $X$.
See B.~Sagan's article~\cite{Sagan2011} for a survey of various types of CSP instances.

Many known instances of the cyclic sieving phenomenon involve a
set $X$ whose size is a Catalan number.
Once such a CSP triple is obtained, one can ask if $X$ can be partitioned $X = \sqcup_j X_j$ in
such a way that the group action on $X$ induces a group action on $X_j$ for all $j$, and,
in that case, also ask if there is a \defin{refinement} of the CSP triple in question.

\begin{definition}[Refinement of cyclic sieving]
The family $\{(X_j, C_n, f_j(q))\}_j$ of CSP triples is said to
\defin{refine} the CSP triple $(X, C_n, f(q))$ if
\begin{itemize}
	\item $\bigsqcup X_j = X$,
	\item $\sum_j f_j(q) = f(q)$ and
	\item the $C_n$-action on $X_j$ coincides with the $C_n$-action on $X$ restricted to $X_j$, for all $j$.
\end{itemize}
\end{definition}
Typically, the sets $X_i$ are of the form $X_j=\{ x \in X: \st(x)=j \}$ for
some statistic $\st:X \to \setN$ that is preserved by the group action.
Examples of such statistics are the number of cyclic descents of a word,
the number of blocks of a partition, and the number of ears of a triangulation of
an $n$-gon---all with the group action being (clockwise) cyclic rotation.
Throughout the paper, \emph{we shall consistently use
the order of the group (or group generator) as subscript}.
For example, rotation by $2\pi/n$ is denoted $\rot_n$.

For $s \in \{0,n\}$, the $q$-analog in \eqref{eq:introqInterpolation} admits a natural refinement,
so that the type $A$ and type $B$ $q$-Narayana polynomials are recovered.
The $q$-Narayana polynomials can be used to refine the aforementioned instances of the CSP.
It is therefore natural to ask if there is a $q$-analog of \eqref{eq:introInterpolation}
for arbitrary $s \in \{0,1,\dotsc,n\}$ which also exhibits similar combinatorial properties
as the type $A$ and type $B$ $q$-Narayana polynomials.
We discuss partial results and motivations behind this problem in \cref{sec:typeABNarayanaQuest}.

In the process of analyzing this intriguing question, we discovered
several new instances of the cyclic sieving phenomenon.
Some concern new $q$-analogs of Catalan numbers, while others refine known instances.
In the tables in \cref{sec:catalanObjects}, we present a
comprehensive (but most likely incomplete) overview of the current
state-of-the-art regarding the cyclic sieving phenomenon
involving Catalan and Narayana objects of type $A$ and $B$.

\subsection{Overview of our results}

We only highlight some of the results in our paper; in addition we also
prove several other results which fill gaps in the literature.
In \cref{sec:s=0}, the main result is the following theorem, which
is a new refined CSP instance on Catalan objects.
It can be stated either in terms of promotion (denoted $\prom_{2n}$)
on two-row standard Young tableaux with $k$
cyclic descents, $\SYT_{\cdes}(n^2,k)$, or non-crossing perfect matchings with $k$ short edges, $\NCM_{\sh}(n, k)$.
\begin{theorem}[\cref{thm:PMrefinedCSP}]
Let $k, n \geq 2$ be natural numbers and let
\[
 \qSYT[q]{n}{k} \coloneqq \frac{q^{k(k-2)}(1+q^n)}{[n+1]_q} \qbinom{n+1}{k} \qbinom{n-2}{k-2}.
\]
Then \[\sum_{k} \qSYT[q]{n}{k} = \Cat[q]{n},\] and
the triples
\[
(\SYT_{\cdes}(n^2,k),\langle \prom_{2n} \rangle,\qSYT[q]{n}{k})\]
and
\[
(\NCM_{\sh}(n, k),\langle \rot_{2n} \rangle, \qSYT[q]{n}{k})
\]
exhibit the cyclic sieving phenomenon.
\end{theorem}

In \cref{sec:s=1}, we study the set of so-called non-crossing (1,2)-configurations on $n$
vertices, which we denote by $\NCC(n+1)$.
The cardinality of this set is the Catalan number $\Cat{n+1} = \binom{2n}{n}-\binom{2n}{n-2}$.
We define a simple ``rotate-and-flip'' action on $\NCC(n+1)$ which has order $2n$
and is reminiscent of promotion.
\begin{theorem}[\cref{thm:newCatalanCSP}]\label{thm:nccTwistCSP}
The triple
\[
\left(
\NCC(n+1), \langle \twist_{2n} \rangle, \qbinom{2n}{n} - q^2 \qbinom{2n}{n-2}
\right)
\]
exhibits the cyclic sieving phenomenon.
\end{theorem}
Note that we use a quite non-standard $q$-analog of the Catalan numbers here,
which has not appeared in the context of cyclic sieving before.
Cyclic sieving on non-crossing (1,2)-configurations was studied earlier by M.~Thiel~\cite{Thiel2017}, with rotation as the group action.
In \cref{thm:thielRefinement} and \cref{cor:thielRefinementNarayana},
we refine Thiel's result. In particular, we obtain a new CSP instance involving the
$q$-Narayana polynomial $\Nar[q]{n+1}{k}$.

In \cref{sec:s=n}, we study various instances of cyclic sieving involving the
type~$B$~Catalan numbers, $\binom{2n}{n}$. Some results have more or less appeared in earlier works, but we make some of the results more explicit.
One novel result is a type~$B$ version of \cref{thm:nccTwistCSP},
where we consider the twist action on \emph{type~$B$~non-crossing (1,2)-configurations}.
Briefly, such objects are obtained from elements in $\NCC(n)$ by choosing to mark one edge.
\begin{theorem}[\cref{thm:typeB-NCC-Twist}]
The triple
\[
\left(
\NCC^B(n+1), \langle \twist^2_{2n} \rangle, \qbinom{2n}{n}
\right)
\]
exhibits the cyclic sieving phenomenon.
\end{theorem}
As in type $A$, we also obtain a refined cyclic sieving result
in \cref{thm:typeB-NCC-rot-CSP} where we consider rotation instead.

In \cref{sec:branden}, we briefly consider two-column semistandard Young tableaux, and note in
\cref{thm:ssytNarayanaCSP} that $(\SSYT(2^k,n),\langle\kprom_n\rangle, \Nar[q]{n+1}{k+1})$
is a CSP triple, where $\kprom_n$ denotes the so-called $k$-promotion and
$\SSYT(2^k,n)$ is the set of semistandard Young tableaux of
the rectangular shape $2^k$ whose maximal entry is at most $n$.

In \cref{sec:earRefinement}, we refine the classical CSP triple on triangulations of
an $n$-gon by taking \emph{ears} into consideration. An ear in a triangulation
is a triangle formed by three cyclically consecutive vertices.
We let $\TRI_{\ear}(n,k)$ denote the set of triangulations
of an $n$-gon with $k$ ears.

\begin{theorem}[\cref{thm:earTriRefinement} and \cref{thm:refinedTriCSP}]
Let $2 \leq k \leq \frac{n}{2}$ and let
\begin{equation*}
\Tri[q]{n}{k} \coloneqq
q^{k(k-2)} \frac{[n]_q}{[k]_q} \qbinom{n-4}{2k-4}\Cat[q]{k-2}
\left( \sum_{j=0}^{n-2k} q^{j(n-2)}\qbinom{n-2k}{j}\right).
\end{equation*}
Then \[\sum_k \Tri[q]{n}{k} = \Cat[q]{n-2},\] and
\[
\left(
\TRI_{\ear}(n,k), \langle \rot_n \rangle, \Tri[q]{n}{k}
\right)
\]
exhibits the cyclic sieving phenomenon.
\end{theorem}

In the last section, we consider another natural interpolation between
type $A$ and type $B$ Catalan objects and prove a
cyclic sieving result using standard methods.

Finally, a word about the proofs in this paper. There are traditionally two different approaches
to proving instances of the cyclic sieving phenomenon --- combinatorial\footnote{Or ``brute-force''.}
or representation-theoretical (using vector spaces and diagonalization).
In this paper we exclusively use the combinatorial approach, meaning that we need to explicitly evaluate the CSP-polynomials at roots of
unity and also count the number fixed points of the sets under the group actions.
It may also involve the use of equivariant bijections to derive new CSP triples from the previously known ones.

\section{Preliminaries}

We shall use standard notation in the area of combinatorics, see the go-to references
\cite{StanleyEC2,Macdonald1995}. In particular, $\mdefin{[n]} \coloneqq \{1,2,\dotsc,n\}$
and it should not be confused with the $q$-analog $[n]_q$ defined further down.

\subsection{Words and paths}

Given a word $w=w_1 \dotsm w_n \in [k]^n$, a \defin{descent} is an index $i \in [n-1]$
such that $w_i > w_{i+1}$. We let the \defin{major index}, denoted $\mdefin{\maj(w)}$,
be the sum of the descents of $w$. An \defin{inversion} in $w$
is a pair of indices $i,j \in [n]$ such that $i<j$ and $w_i>w_j$.
We let $\mdefin{\inv(w)}$ be the number of inversions of $w$.
Let $\mdefin{\BW(n,k)}$ denote the set of binary words of length $n$ with exactly $k$ ones.

Let $\mdefin{\PATHS(n)}$ be the set of paths from $(0,0)$ to $(n,n)$
using north, $(1,0)$, and east, $(0,1)$, steps.
A \defin{peak} is a north step followed by an east step, and a \defin{valley} is an east step followed by a north step.
We have an obvious bijection $\PATHS(n) \leftrightarrow \BW(2n,n)$
where we identify north steps with zeros.
Given $P \in \PATHS(n)$, we let $\mdefin{\maj(P)}$ be
defined as the sum of the positions of the valleys of the path $P$.
Observe that this coincides with the major index of the corresponding binary word,
as valleys correspond to descents.
We shall also let $\mdefin{\pmaj(P)}$ denote the sum of the positions of
the \emph{peaks}.
For a path $P \in \PATHS(n)$, we let the \defin{depth}, $\mdefin{\dep(P)}$ be the largest
value of $r \geq 0$ such that the path touches the line $y=x-r$.
Let us define $\mdefin{\PATHS_{s}(n)} \subseteq  \PATHS(n)$ as the set of paths with
$\dep(P)\leq s$. We set $\mdefin{\DYCK(n)}\coloneqq \PATHS_0(n)$.

\subsection{\texorpdfstring{$q$}{q}-analogs}

Roughly, a $q$-analog of a certain expression is a rational function in the variable $q$
from which we can obtain the original expression in the limit $q \to 1$.
\begin{definition}
	Let $n \in \setN$. Define the \defin{$q$-analog of $n$} as
	$\mdefin{[n]_q} \coloneqq 1+q+\dotsb+q^{n-1}$.
	Furthermore, define the \defin{$q$-factorial of $n$} as
	$
	\mdefin{[n]_q!} \coloneqq [n]_q[n-1]_q \dotsm [1]_q
	$.
	Lastly, the \defin{$q$-binomial coefficient} is defined as
	\[
	\mdefin{\qbinom{n}{k}} \coloneqq
	\dfrac{[n]_q!}{[n-k]_q![k]_q!} = \sum_{b \in \BW(n,k)}q^{\inv(b)} = \sum_{b \in \BW(n,k)}q^{\maj(b)}
	\]
	if $n \geq k\geq 0$, and $\qbinom{n}{k}\coloneqq 0$ otherwise.
	Note that the $q$-binomial coefficients are polynomials in the variable $q$,
	see \cite{StanleyEC1} for more background.
	The \defin{$q$-multinomial coefficients} are defined in a similar manner.
\end{definition}

\begin{theorem}[$q$-Vandermonde identity]\label{thm:q-vandermonde}
 The \defin{$q$-Vandermonde identity} states that for non-negative integers $a$, $b$, $c$,
 we have that
 \begin{equation}
 \qbinom{a+b}{c}  =  \sum_{j} q^{j(a-c+j)} \qbinom{a}{c-j} \qbinom{b}{j}.
 \end{equation}
\end{theorem}

\begin{theorem}[$q$-Lucas theorem, see e.g.~\cite{Sagan1992}]\label{thm:q-Lucas}
	Let $n, k\in \setN$. Let $n_1, n_0, k_1, k_0$ be the unique
	natural numbers satisfying $0 \leq n_0, k_0 \leq d-1$ and $n=n_1d+n_0$, $k=k_1d+k_0$. Then
	\[
	\qbinom{n}{k} \equiv \binom{n_1}{k_1}\qbinom{n_0}{k_0} \pmod{\Phi_d(q)}
	\]
	where $\Phi_d(q)$ is the $d^\thsup$ cyclotomic polynomial. In particular, we have
	\begin{equation}
	\qbinom[\xi]{n}{k} = \binom{n_1}{k_1}\qbinom[\xi]{n_0}{k_0}
	\end{equation}
	if $\xi$ is a primitive $d^\thsup$ root of unity.
\end{theorem}

When $\xi$ is a root of unity, let $\mdefin{\order}(\xi)$ denote the smallest
positive integer with the property that $\xi^{\order(\xi)}=1$.
The following is a standard lemma that should not need a proof.
\begin{lemma}\label{lem:qFraction}
Let $n, k, d\in \setN$ and let $\xi$ be a primitive $n^\thsup$ root of unity. Then
\[
\lim_{q\to \xi^d}\frac{[n]_q}{[k]_q}=\begin{cases}
n/k & \text{if } \order(\xi^d) \mid k, \\[1em]
0 & \text{otherwise}.
\end{cases}
\]
\end{lemma}
We will use \cref{thm:q-Lucas} and \cref{lem:qFraction} in later sections.

\begin{lemma}\label{lem:cspSimple}
Let $\xi$ be a primitive $n^\thsup$ root of unity,
and suppose that $f \in \setN[q]$ is such that $f(\xi^j) \in \setZ$ for all $j \in \setZ$.
Then for all $j \in \setZ$, $f(\xi^j) = f(\xi^{\gcd(j,n)})$.
\end{lemma}
\begin{proof}
In \cite[Lem.~2.2]{AlexanderssonAmini2018},
it is proved that $f$ (up to mod $q^n-1$) is a linear combination of
\[
 h_d(q) \coloneqq \sum_{i=0}^{n/d-1} q^{di} = \frac{[n]_q}{[d]_q} \qquad \text{ where } d \mid n.
\]
It then suffices to verify that
\[
h_d(\xi^j) = h_d( \xi^{\gcd(j,n)})
=
\begin{cases}
\frac{n}{d} &\text{ if } \frac{n}{\gcd(j,n)} \mid d, \\
0 &\text{ otherwise}
\end{cases}
\]
for all $d \mid n$, $j\in \setZ$, which is straightforward by using \cref{lem:qFraction}.
\end{proof}
Hence, if we know that $f(\xi^d) \in \setZ$ for all $d \in \setZ$,
it suffices to verify \eqref{eq:cspDef} for all $d \mid n$.
There is a related result about computing the number of fixed points.
\begin{lemma}\label{lem:cspSimple2}
	Suppose that $C_n = \langle g \rangle$ acts on the set $X$. If $d \in \setZ$, then
	\[|\{ x\in X : g^d \cdot x = x \}| = |\{ x\in X : g^{\gcd(n,d)} \cdot x = x \}|\]
\end{lemma}
\begin{proof}
Note that all elements of $C_n$ with order $o$ generate the same subgroup $S \subseteq C_n$. If $h, h'$ are both of order $o$, then $\langle h \rangle = \langle h' \rangle = S$, and $h \cdot x = x$ implies that $h' \cdot x = h^e \cdot x = x$, for some $e \in \setZ$.
\end{proof}
\Cref{lem:cspSimple} and \Cref{lem:cspSimple2} are useful facts and are used implicitly in many papers.
We shall use them without further mention throughout the paper.

\subsection{Catalan and Narayana numbers}

The \defin{Catalan numbers} $\Cat{n} \coloneqq \frac{1}{n+1}\binom{2n}{n}$ are indexed by natural numbers.
These numbers occur frequently in combinatorics, see \oeis{A000108} in the OEIS,
and give the cardinalities of many families of combinatorial objects.
For the purpose of this paper, we note that the following sets all have cardinality $\Cat{n}$.
\begin{itemize}
\item $\DYCK(n)$: the set of Dyck paths of size $n$, that is, the subset of paths in
$\PATHS(n)$ which never touch the line $y=x-1$,
\item $\SYT(n^2)$: the set of standard Young tableaux with two rows of length $n$,
\item $\NCP(n)$: the set of non-crossing partitions on $n$ vertices,
\item $\NCM(n)$: the set of non-crossing matchings on $2n$ vertices,
\item $\TRI(n)$: the set of triangulations of an $(n+2)$-gon,
\item $\NCC(n)$: the set of non-crossing $(1,2)$-configurations on $n-1$ vertices.
\end{itemize}
Examples of such objects are listed in \cref{sec:catalanObjectsAppendix}.

Throughout this paper, we use \defin{MacMahon's $q$-analog} of the Catalan numbers.
For any natural number $n$, the $n^\thsup$ $q$-Catalan number is defined by
\begin{align}\label{eq:qCatalanDef}
\mdefin{\Cat[q]{n}} &\coloneqq \frac{1}{[n+1]_q} \qbinom{2n}{n} =
\qbinom{2n}{n}-q\qbinom{2n}{n-1} \\
&=\sum_{P \in \DYCK(n)}q^{\maj(P)} = \sum_{T \in \SYT(n^2)}q^{\maj(T)-n}.
\end{align}
A definition of $\maj$ on standard Young tableaux can be found in the next section.
The \defin{Narayana numbers}
$\mdefin{\Nar{n}{k}} \coloneqq \frac{1}{n}\binom{n}{k}\binom{n}{k-1}$, indexed by two natural numbers $n$ and $k$ such that $1 \leq k \leq n$, are also well-known and have many applications, see the OEIS entry \oeis{A001263}.
The Narayana numbers \defin{refine} the Catalan numbers in the sense
that $\sum_k \Nar{n}{k} = \Cat{n}$.
For our purposes, it suffices to know that the following sets all have cardinality $\Nar{n}{k}$.
\begin{itemize}
\item $\mdefin{\DYCK(n, k)}$: the set of paths in $\DYCK(n)$ with exactly $k$ peaks,
\item $\mdefin{\SYT(n^2, k)}$: the set of tableaux in $\SYT(n^2)$ with exactly $k$ descents,
\item $\mdefin{\NCP(n, k)}$: the set of partitions in $\NCP(n)$ with exactly $k$ blocks,
\item $\mdefin{\NCC(n, k)}$: the set of non-crossing $(1,2)$-configurations in $\NCC(n)$
such that the numbers of proper edges plus the number of loops is equal to $k-1$,
\item $\mdefin{\NCM(n, k-1)}$: the set of non-crossing matchings in $\NCM(n)$
with $k-1$ even edges.
\end{itemize}
The \defin{$q$-Narayana numbers} are defined as the $q$-analog
\begin{align*}
\mdefin{\Nar[q]{n}{k}} \coloneqq \frac{q^{k(k-1)}}{[n]_q}\qbinom{n}{k}\qbinom{n}{k-1} &=
\sum_{P \in \DYCK(n,k)} q^{\maj(P)}\\ &= \sum_{T \in \SYT(n^2,k)}q^{\maj(T)-n}.
\end{align*}
The $q$-Narayana numbers refine the $q$-Catalan numbers, that is,
$\sum_k \Nar[q]{n}{k}$ is equal to $\Cat[q]{n}$.
We also mention that there is a bijection $\NCPtoDYCK$
from $\NCP(n,k)$ to $\DYCK(n,k)$ described in \cref{bij:NCPtoDyck}.
Thus,
\begin{equation}\label{eq:narayanaFromNCP}
 \Nar[q]{n}{k} = \sum_{\pi \in \NCP(n,k)} q^{\maj(\NCPtoDYCK(\pi))}.
\end{equation}
For more background, see \cite{Simion1994} and \cite{ZhaoZhong2011}.

\subsection{Type \texorpdfstring{$B$}{B} Catalan numbers}\label{sec:PrelTypeB}

We shall now describe the type $B$ analogs of the combinatorial objects we saw in the previous section.
The \defin{type $B$ Catalan numbers} $\mdefin{\CatB{n}}$ are defined as
\begin{equation}\label{eq:typeBCatalanDef}
 \mdefin{\CatB{n}} \coloneqq \binom{2n}{n} = \sum_{k=0}^n \binom{n}{k} \binom{n}{k}.
\end{equation}
The \defin{type $B$ Narayana numbers} $\mdefin{\NarB{n}{k}}$ are defined as
\begin{equation}
\mdefin{\NarB{n}{k}} \coloneqq \binom{n}{k}^2.
\end{equation}
The type $B$ Narayana numbers clearly refine the $B$ Catalan numbers, as can be seen from \eqref{eq:typeBCatalanDef}.
Among other things, they count the number of elements in $\PATHS(n)$ with $k$ valleys.
For a more comprehensive list, see \oeis{A008459} in the OEIS, and also the reference \cite{Armstrong2009} for more background.
The \defin{$q$-analogs of the type $B$ Catalan numbers} and the \defin{type $B$ $q$-Narayana numbers} are defined as
\begin{equation}\label{eq:typeBCatAndNar}
 \mdefin{\CatB[q]{n}} \coloneqq \qbinom{2n}{n},
 \qquad \mdefin{ \NarB[q]{n}{k} } \coloneqq q^{k^2} \qbinom{n}{k} \qbinom{n}{k}.
\end{equation}
It is straightforward to verify that $\CatB[q]{n} = \sum_{k=0}^n \NarB[q]{n}{k}$.
Moreover, one can show that
\begin{align}\label{eq:B Narayana formula}
  \sum_{k=0}^n t^k \NarB[q]{n}{k} &= \sum_{T \in \SYT((2n,n)/(n))} t^{|\DES(T)|}q^{\maj(T)} \\
  &= \sum_{P \in \PATHS(n)} t^{\valleys(P)} q^{\maj(P)},
\end{align}
see \cite{Sulanke1998,Sulanke2002}.

The following combinatorial interpretation of the type $B$ $q$-Narayana numbers
is mentioned in I.~Macdonald's book~\cite[p. 400]{Macdonald1995}.
Let $V$ be a $2n$-dimensional vector space over $F_q$,
and let $U$ be an $n$-dimensional subspace of $V$.
Then $\NarB[q]{n}{k}$ is the number of $n$-dimensional subspaces $U'$ of $V$
such that $\dim(U\cap U')=n-k$.

\subsection{Overview of the CSP on Catalan and Narayana objects}\label{sec:catalanObjects}

\cref{Table:typeA} and \cref{Table:typeB} list the state-of-the-art of the CSP on
Catalan-type objects of type $A$ and $B$ respectively,
including the results proven in the present paper.
Examples of such objects can be found in \cref{sec:catalanObjectsAppendix}.
We use several bijections (described in \cref{sec:bijections})
between Catalan and Narayana objects, see \cref{fig:Catalan zoo}.

\begin{table}[!ht]
 \centering
\begin{tabular}{p{0.6\textwidth}p{0.3\textwidth}}
\toprule
\textbf{Type $A$ set \& reference} & \textbf{Group \& polynomial} \\
\midrule
Triangulations of an $n$-gon &  Rotation $\rot_{n}$ \\
\cite[Thm.~7.1]{ReinerStantonWhite2004} & $\Cat[q]{n-2}$ \\
Triangulations of an $n$-gon with $k$ ears & Rotation $\rot_{n}$ \\
\cref{thm:refinedTriCSP} &  Complicated \\
\midrule
\midrule
Two-row standard Young tableaux & Promotion $\prom_{2n}$ \\
\cite[Thm.~1.3]{Rhoades2010} & $\Cat[q]{n}$ \\
Non-crossing matchings & Rotation $\rot_{2n}$ \\
See \cite[Thm.~5]{Heitsch2007} and \cite[Thm.~8.3]{Rhoades2010}. & $\Cat[q]{n}$ \\
%
Non-crossing partitions & Kreweras compl., $\krew_{2n}$ \\
See \cite[Thm.~1]{Heitsch2007}. & $\Cat[q]{n}$ \\
\midrule
Non-crossing matchings with $k$ short edges & Rotation $\rot_{2n}$ \\
\cref{thm:PMrefinedCSP} & $\frac{q^{k(k-2)}(1+q^n)}{[n+1]_q} \qbinom{n+1}{k} \qbinom{n-2}{k-2}$ \\
Two-row SYT with $k$ cyclic descents & Promotion $\prom_{2n}$ \\
\cref{thm:PMrefinedCSP} & $\frac{q^{k(k-2)}(1+q^n)}{[n+1]_q} \qbinom{n+1}{k} \qbinom{n-2}{k-2}$ \\

\midrule
Non-crossing partitions with $k$ parts & Rotation $\rot_n$ \\
\cite[Thm. 7.2]{ReinerStantonWhite2004} & $\Nar[q]{n}{k}$ \\

Non-crossing matchings with $k$ even edges & Rotation $\rot_n$ \\
\cref{prop:NCMNarayanaCSP} & $\Nar[q]{n}{k}$ \\
\midrule
\midrule
Non-cross.~(1,2)-config.  &  Rotation $\rot_n$ \\
\cite{Thiel2017} & $\Cat[q]{n+1}$ \\
Non-cross.~(1,2)-config.~with $l$ loops and $e$ edges & Rotation $\rot_n$ \\ \cref{thm:thielRefinement} &
 $  \frac{q^{e(e+1)+(n+1)l}}{[e+1]_q}\qbinom{n}{e,e,l,n-2e-l}$ \\
Non-cross.~(1,2)-config.~with $k$ edges or loops & Rotation $\rot_n$  \\ \cref{cor:thielRefinementNarayana} & $\Nar[q]{n+1}{k}$ \\
Two-column SSYT, $\SSYT(2^k,n)$ & $k$-promotion $\kprom_n$ \\
\cref{thm:ssytNarayanaCSP} &  $\Nar[q]{n+1}{k+1}$. \\
 \midrule
Non-cross.~(1,2)-config. &  Twisted rotation $\twist_{2n}$  \\
\cref{thm:newCatalanCSP} & $\qbinom{2n}{n} - q^2\qbinom{2n}{n-2} $ \\
\bottomrule
\end{tabular}
\caption{
The current state-of-the-art regarding cyclic sieving on type $A$ Catalan objects,
including the new results presented in this article.
}\label{Table:typeA}
\end{table}
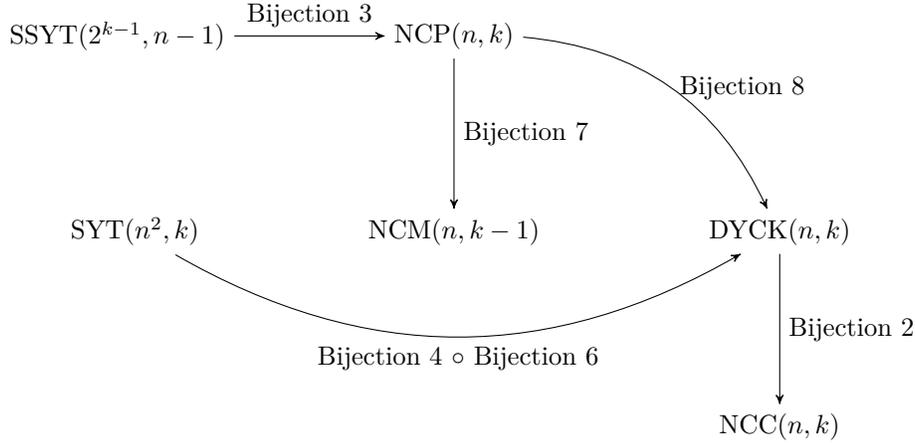
\begin{figure}[ht!]
	\begin{tikzpicture}[
	thin arrow/.style = {
		->, -stealth',
	},
 	node distance=2cm and 2cm,
	]
	[align=center,node distance=4cm]
	\node (syt) {$\SYT(n^2,k)$};
	\node[right = of syt] (ncm) {$\NCM(n,k-1)$};
	\node[right = of ncm] (dyck) {$\DYCK(n,k)$};
	\node[above = of ncm] (ncp) {$\NCP(n,k)$};
	\node[below = of dyck] (ncc) {$\NCC(n,k)$};
	\node[left = of ncp] (ssyt) {$\SSYT(2^{k-1},n-1)$};

	\draw[thin arrow] (ncp) edge node[midway,right] {\cref{bij:NCPtoNCM}} (ncm);
	\draw[thin arrow] (ncp) edge[bend left] node[midway,right] {\cref{bij:NCPtoDyck}} (dyck);
	\draw[thin arrow] (dyck) -- (ncc) node[midway,right] {\cref{def:laserBij}};
	\draw[thin arrow] (ssyt) -- (ncp) node[midway,above] {\cref{bij:SSYTtoNCP}};
	\draw[thin arrow] (syt) edge[bend right] node[midway,below] {\cref{bij:NCMtoDyck}  $\circ$ \cref{bij:SYTtoNCM}} (dyck);
	\end{tikzpicture}
	\caption{Schematic overview of the Narayana zoo. Note that we have bijections from $\SYT(n^2) \rightarrow \NCM(n)$ (\cref{bij:SYTtoNCM}) and $\NCM(n) \rightarrow \DYCK(n)$ (\cref{bij:NCMtoDyck}), but they do not respect the particular Narayana refinements.}
	\label{fig:Catalan zoo}
\end{figure}

The bijections in \cref{fig:Catalan zoo} respect a Narayana refinement
and so, for example, $\SYT(n^2, k)$, $\NCM(n,k-1)$ and $\NCP(n,k)$ are all equinumerous.
Furthermore, composing the natural bijections \cref{bij:SYTtoNCM}
and the inverse of \cref{bij:NCPtoNCM}, we get that
promotion on SYT corresponds to rotation on non-crossing matchings.

However, promotion on standard Young tableaux does not preserve the number of descents,
but rotation preserves the number of even edges of matchings.
It follows that the cyclic sieving phenomenon on non-crossing
matchings with a specified number of even edges
does not correspond to one on $\SYT(n^2)$ with
a fixed number of descents with promotion as the action.
In general, \emph{a specific Narayana refinement might be incompatible with
a cyclic group action}. By \cref{bij:SYTtoNCM}, here the compatible statistic
one should use for $\SYT(n^2)$ is the number of even entries in the first row.

A note on the general philosophy of the paper.
Having many different sets of objects and well-behaved bijections between
these sets turns out to be a very fruitful approach to proving instances of the CSP.
In this context, well-behaved often times means that the bijection is equivariant.
If a group action looks complicated on a certain set, it can perhaps be made easier if
one first applies an equivariant bijection and then studies the image.
For example, promotion on $\SYT(n^2)$ is complicated while rotation on $\NCM(n)$ is easier.
There is a type of converse of the above. If one has two different CSP triples
with identical CSP-polynomials and whose cyclic groups have the same order, then
there exists an equivariant bijection between these two sets
(by sending orbits to orbits of the same size).

\begin{table}[!ht]
 \centering
\begin{tabular}{p{0.6\textwidth}p{0.3\textwidth}}
\toprule
\textbf{Type $B$ set \& reference} & \textbf{Group \& polynomial} \\
\midrule
Binary words $\BW(2n,n)$  &  Cyclic shift $\shift_{2n}$  \\
\cite[Prop.~4.4]{ReinerStantonWhite2004} & $ \qbinom{2n}{n} $ \\
Skew two-row SYT, $\SYT((2n,n)/(n))$  & Promotion $\prom_{2n}$ \\
\cite[Section 3.1]{StrikerWilliams2012} & $ \qbinom{2n}{n} $ \\
Type $B$ root poset order ideals $\mathrm{OI}_B(n)$ & Rowmotion, $\rowmotion_{2n}$ \\
\cite[Thm.~1.5]{ArmstrongStumpThomas2013} and \cite{StrikerWilliams2012} & $\qbinom{2n}{n} $  \\
Type $B$ non-crossing partitions $\NCPB(n)$  &  Rotation $\rot_{2n}$ \\
\cite[Thm.~1.5]{ArmstrongStumpThomas2013} & $\qbinom{2n}{n} $ \\
Marked $(1,2)$-configs. &  Twisted rotation $\twist^2_{2n}$  \\
\cref{thm:typeB-NCC-Twist} & $\qbinom{2n}{n} $ \\
\midrule
Binary words $\BW(2n,n)$ with $k$ cyclic descents & Cyclic shift $\shift_{2n}$ \\
\cref{prop:typeBCdesCSP}, \cite[Thm.~1.5]{AhlbachSwanson2018} & $q^{k(k-1)}(1+q^n)\qbinom{n}{k}\qbinom{n-1}{k-1}$ \\
\midrule
Marked $(1,2)$-configs.~with $e$ edges and $l$ loops & Rotation $\rot_{n}$  \\
\cref{thm:typeB-NCC-rot-CSP} & Complicated \\
Marked $(1,2)$-configs.~with $k$ edges and loops & Rotation $\rot_{n}$  \\
\cref{cor:typeB-NCC-rot} & Complicated \\
Marked $(1,2)$-configs. &  Rotation $\rot_{n}$  \\
\cref{cor:typeB-NCC-rot} & Complicated \\
\midrule
Type $B$ NCP with $2k$ or $2k+1$ blocks & Rotation $\rot_{n}$ \\
\eqref{eq:TypeB-NCM-NarayanaBlock} & $ q^{k^2}\qbinom{n}{k}[2]$ \\
\midrule
Type $B$ triangulations on $2n+2$ vertices & Rotation $\rot_{n+1}$ \\
See \cite[Thm.~4.1]{EuFu2008}. & $\qbinom{2n}{n}$ \\
\bottomrule
\end{tabular}
\caption{
The state-of-the-art regarding cyclic sieving on type $B$ Catalan objects,
including the new results presented in this article.}\label{Table:typeB}
\end{table}

\section{Type \texorpdfstring{$A/B$}{A/B}-Narayana numbers and a quest for a \texorpdfstring{$q$}{q}-analog}\label{sec:typeABNarayanaQuest}

We shall now discuss a natural interpolation between type $A$ and type $B$ Catalan numbers.
The following observation illustrates this interpolation.
For any $s \geq 0$, the sets below are equinumerous:
\begin{enumerate}
 \item the set of skew standard Young tableaux $\SYT((n+s,n)/(s))$,
 \item the set of lattice paths, $\PATHS_{s}(n)$,
 \item the set of order ideals in the type $B$ root poset with at most $s$ elements on the top diagonal.
\end{enumerate}
Note that for $s=0$, we recover sets of cardinality $\Cat{n}$,
and for $s=n$, we recover sets of cardinality $\CatB{n}$.
Bijective arguments are given below in \cref{prop:SYTPathBij} and \cref{prop:rootIdealsSkewSYTBij}.

Let $T\in \SYT(\lambda/\mu)$ where the diagram of $\lambda/\mu$ has $n$ boxes.
A \defin{descent} of $T$ is an integer $j \in \{1,\dotsc,n-1\}$
such that $j+1$ appears in a row below $j$.
The \defin{major index} of $T$ is the sum of the descents.
The major-index generating function for skew standard Young tableaux is defined as
\begin{equation}\label{eq:qanalogSYT}
\mdefin{f^{\lambda/\mu}(q)} \coloneqq \sum_{T \in \SYT(\lambda/\mu)} q^{\maj(T)}
\end{equation}
when $\lambda/\mu$ is a skew shape.
Our motivation for studying this polynomial is \cite[Thm.~46]{AlexanderssonPfannererRubeyUhlin2020x} which
states that for any skew shape $\lambda/\mu$ where each row contains a multiple of $m$ boxes,
there must exist some cyclic group action $C_m$ of order $m$ such that
\begin{equation}\label{eq:stretchedSYTCSP}
 \left( \SYT(\lambda/\mu), C_m,
  f^{\lambda/\mu}(q) \right)
\end{equation}
is a CSP triple. We do not know how such a group action looks like except in the case $m=2$.
In that case one can use \defin{evacuation}, defined by Sch\"utzenberger \cite{Schutzenberger1963}.

\subsection{Skew standard Young tableaux with two rows}

We now describe a bijection between skew SYT with two rows and certain lattice paths.
\begin{proposition}\label{prop:SYTPathBij}
Given $s \in \{0,\dotsc,n\}$, there is a bijection
\[
 \SYT((n+s,n)/(s)) \longrightarrow \PATHS_s(n)
\]
which sends descents in the tableau to peaks in the path.
\end{proposition}
\begin{proof}
A natural generalization of the standard bijection works:
an $i$ in the upper or lower row corresponds to the $i^\thsup$
step in the path being north or east, respectively.
Evidently, a descent in the tableau is sent to a peak in the path.
\end{proof}
Recall that for a SYT $T$ the statistic $\maj(T)$ is the sum of the position of
the descents, which is then sent to $\pmaj(P)$ which is the sum of the
positions of the peaks in the corresponding path $P$.

Let
\begin{align*}
\mdefin{X_{n,s}(q)} \coloneqq \sum_{P \in \PATHS_{s}(n)} q^{\pmaj(P)}
\qquad
\text{ and }
\qquad
\mdefin{Y_{n,s}(q)} \coloneqq \sum_{P \in \PATHS_{s}(n)} q^{\maj(P)}.
\end{align*}
By \cref{prop:SYTPathBij} we also have $X_{n,s}(q) = f^{(n+s,n)/(s)}(q)$.

\begin{proposition}[{\cite[Thm.~7]{Krattenthaler1989}}]\label{prop:qIdentsForSkewSYT}
For $n\geq 1$ and $n\ge s\geq 0$,
\begin{equation}\label{eq:shiftedPaths}
X_{n,s}(q) = \qbinom{2n}{n} - \qbinom{2n}{n-s-1}
\;
\text{ and }
\;
Y_{n,s}(q) = \qbinom{2n}{n} - q^{s+1} \qbinom{2n}{n-s-1}.
\end{equation}
In particular,
\[
q^{-n}X_{n,0}(q) = Y_{n,0}(q) = \Cat[q]{n} \text{ and }
X_{n,n}(q) = Y_{n,n}(q) = \CatB[q]{n}.
\]
\end{proposition}

\subsection{Root lattices in type \texorpdfstring{$A/B$}{A/B}}

The following illustrate the root ideals of $B_n$ where $n=3$.
There are in total $\binom{2\cdot 3}{3}=20$ such ideals.
A root ideal is simply a lower set in the root poset---marked as shaded boxes in the diagrams below.
Root ideals are also called \defin{non-nesting partitions of type $W$}, where $W$
is the Weyl group of some root system.

\begin{align*}
\ytableausetup{boxsize=0.9em}
\ytableaushort{{*(lightgray)},{*(lightgray)}{*(lightgray)},{*(lightgray)}{*(lightgray)}{*(lightgray)},{*(lightgray)}{*(lightgray)},{*(lightgray)}} \;
\ytableaushort{{*(white)},{*(lightgray)}{*(lightgray)},{*(lightgray)}{*(lightgray)}{*(lightgray)},{*(lightgray)}{*(lightgray)},{*(lightgray)}} \;
\ytableaushort{{*(white)},{*(white)}{*(lightgray)},{*(lightgray)}{*(lightgray)}{*(lightgray)},{*(lightgray)}{*(lightgray)},{*(lightgray)}} \;
\ytableaushort{{*(white)},{*(white)}{*(lightgray)},{*(white)}{*(lightgray)}{*(lightgray)},{*(lightgray)}{*(lightgray)},{*(lightgray)}} \;
\ytableaushort{{*(white)},{*(white)}{*(lightgray)},{*(white)}{*(lightgray)}{*(lightgray)},{*(white)}{*(lightgray)},{*(lightgray)}} \;
\ytableaushort{{*(white)},{*(white)}{*(lightgray)},{*(white)}{*(lightgray)}{*(lightgray)},{*(white)}{*(lightgray)},{*(white)}} \;
\ytableaushort{{*(white)},{*(white)}{*(white)},{*(lightgray)}{*(lightgray)}{*(lightgray)},{*(lightgray)}{*(lightgray)},{*(lightgray)}} \;
\ytableaushort{{*(white)},{*(white)}{*(white)},{*(white)}{*(lightgray)}{*(lightgray)},{*(lightgray)}{*(lightgray)},{*(lightgray)}} \;
\ytableaushort{{*(white)},{*(white)}{*(white)},{*(white)}{*(lightgray)}{*(lightgray)},{*(white)}{*(lightgray)},{*(lightgray)}} \;
\ytableaushort{{*(white)},{*(white)}{*(white)},{*(white)}{*(lightgray)}{*(lightgray)},{*(white)}{*(lightgray)},{*(white)}} \\
\ytableaushort{{*(white)},{*(white)}{*(white)},{*(white)}{*(white)}{*(lightgray)},{*(lightgray)}{*(lightgray)},{*(lightgray)}} \;
\ytableaushort{{*(white)},{*(white)}{*(white)},{*(white)}{*(white)}{*(lightgray)},{*(white)}{*(lightgray)},{*(lightgray)}} \;
\ytableaushort{{*(white)},{*(white)}{*(white)},{*(white)}{*(white)}{*(lightgray)},{*(white)}{*(lightgray)},{*(white)}} \;
\ytableaushort{{*(white)},{*(white)}{*(white)},{*(white)}{*(white)}{*(lightgray)},{*(white)}{*(white)},{*(lightgray)}} \;
\ytableaushort{{*(white)},{*(white)}{*(white)},{*(white)}{*(white)}{*(lightgray)},{*(white)}{*(white)},{*(white)}} \;
\ytableaushort{{*(white)},{*(white)}{*(white)},{*(white)}{*(white)}{*(white)},{*(lightgray)}{*(lightgray)},{*(lightgray)}} \;
\ytableaushort{{*(white)},{*(white)}{*(white)},{*(white)}{*(white)}{*(white)},{*(white)}{*(lightgray)},{*(lightgray)}} \;
\ytableaushort{{*(white)},{*(white)}{*(white)},{*(white)}{*(white)}{*(white)},{*(white)}{*(lightgray)},{*(white)}} \;
\ytableaushort{{*(white)},{*(white)}{*(white)},{*(white)}{*(white)}{*(white)},{*(white)}{*(white)},{*(lightgray)}} \;
\ytableaushort{{*(white)},{*(white)}{*(white)},{*(white)}{*(white)}{*(white)},{*(white)}{*(white)},{*(white)}}
\end{align*}

An explicit bijection from the set of skew standard Young tableaux $\SYT((n+s,n)/(s))$
to the root ideals of $B_n$ with at most $s$ elements on the top diagonal is described below.
First, let $\mdefin{\mathrm{OI}(n,s)}$ be the set of root ideals with at most $s$ elements on the top diagonal.

\begin{bijection}\label{bij:PathToRootIdeal}
Let $a_1,a_2,\dotsc,a_n$ be the top row of the skew tableau.
We identify this top row using the bijection in \cref{prop:SYTPathBij} with a path
$\alpha\in \PATHS_{s}(n)$ and get that $\dep(\alpha)=\max_i\{a_i-2i+1\}$.
Let $j$ be the smallest value for which the maximum is obtained, so $\dep(\alpha)=a_j-2j+1$.
We then define the map $\phi$ as changing the step $a_j-1$,
just before reaching maximal depth for the first time, from an east step to a north step.
That is, $\phi(\alpha)=a_1,\dotsc,a_{j-1},a_j-1,a_j,\dotsc,a_n$.
This new path ends at $(n-1,n+1)$ and has depth one less than $\alpha$.
We repeat $\dep(\alpha)$ times and get $\phi^{\dep(\alpha)}(\alpha)$ which ends in
$(n-\dep(\alpha),n+\dep(\alpha))$ and has depth zero.
This path always starts with a north step, and the
boxes below and to the right of it make up a root ideal $o$ in $B_n$ with $\dep(\alpha)$ elements in the top diagonal.
Since $\dep(\alpha)\le s$ this gives $o\in \mathrm{OI}(n,s)$ and the desired map.
See \cref{F:root ideals} for an example.
The inverse $\phi^{-1}$ is easily obtained as follows. Given a root ideal $o$,
let $\beta(o)$ be the north-east path along its boundary, starting with an extra north step.
Now, change the north step of $\beta(o)$ after the last time the path has
reached maximum depth to an east step.
The inverse of the bijection is obtained by iterating $\phi^{-1}$ until the path ends in $(n,n)$.
The map $\phi$ has been used many times before, see e.g.~\cite{AlexanderssonLinussonPotka2019}.
\end{bijection}

\begin{figure}[!ht]
\centering
\[
	\ytableausetup{boxsize=1.27em}
	\ytableaushort{\none \none 2578{10},13469}
	\qquad
	\longrightarrow
	\qquad
	\begin{tikzpicture}[x=1.3em,y=1.3em,baseline=(current bounding box.center)]
	\put(0,-24){
	\draw (0,0) [->,thick,blue] (0,0)--(1,0)--(1,1)--(2,1)--(3,1)--(3,2)--(4,2)--(4,3)--(4,4)--(5,4)--(5,5);
	\put(43,17){$\alpha$};
	\draw (0,0) [->,thick,red] (0,0)--(1,0)--(1,1)--(2,1)--(2,2)--(2,3)--(3,3)--(3,4)--(3,5)--(4,5)--(4,6);
	\put(29,30){$\phi(\alpha)$};
	\draw (0,0) [->,thick,black] (0,0)--(0,1)--(0,2)--(1,2)--(1,3)--(1,4)--(2,4)--(2,5)--(2,6)--(3,6)--(3,7);
	\put(-1,63){$\phi^2(\alpha)$};
	};
	\end{tikzpicture}
	\qquad
	\longrightarrow
	\qquad
	\ytableausetup{boxsize=1.27em}
	\ytableaushort{{*(white)},{*(white)}{*(white)},{*(white)}{*(white)}{*(white)},{*(white)}{*(white)}{*(white)}{*(lightgray)},{*(white)}{*(white)}{*(lightgray)}{*(lightgray)}{*(lightgray)},{*(white)}{*(white)}{*(lightgray)}{*(lightgray)},{*(white)}{*(lightgray)}{*(lightgray)},{*(white)}{*(lightgray)},{*(lightgray)}} \;
	\]
\caption{An example of the bijection: $\phi^2(2,5,7,8,10)=\phi(2,4,5,7,8,10)=1,2,4,5,7,8,10$.}
\label{F:root ideals}
\end{figure}
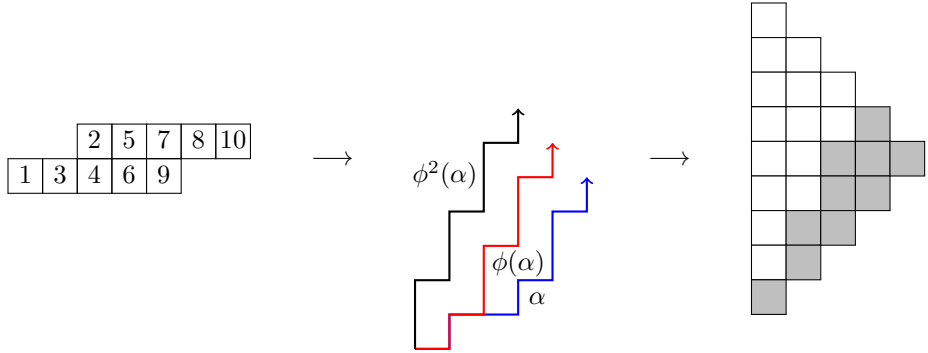

We naturally define $\mdefin{\maj(o)}\coloneqq \maj(\beta(o))$ and
$\mdefin{\pmaj(o)}\coloneqq \pmaj(\beta(o))$ for $o \in \mathrm{OI}(n,s)$.

\begin{proposition}\label{prop:rootIdealsSkewSYTBij}
The map in \cref{bij:PathToRootIdeal} is a bijection so
\[
 |\mathrm{OI}(n,s)| = \binom{2n}{n} - \binom{2n}{n-1-s}.
\]
Furthermore,
\[\sum_{o\in \mathrm{OI}(n,s)} q^{\pmaj(o)}= \qbinom{2n}{n} - \qbinom{2n}{n-s-1}\] and
\[\sum_{o\in \mathrm{OI}(n,s)} q^{\maj(o)}= \qbinom{2n}{n} - q\qbinom{2n}{n-s-1}+\sum_{d=1}^{s}(1-q)\qbinom{2n}{n-d}.\]
\end{proposition}

\begin{proof} The map is clearly a bijection and the first formula follows.
For the second statement note that the map $\phi$ does not change the peaks and thus not $\pmaj$,
but it changes the position of one valley and decreases $\maj$ by 1.
Thus the $q$-polynomial for $\pmaj$ is identical to $X_{n,s}(q)$ in \cref{prop:qIdentsForSkewSYT}.
The $\maj$-generating polynomial for paths with depth at most $s$ is $Y_{n,s}(q)$ by \cref{prop:qIdentsForSkewSYT}.
Thus, for paths having depth exactly $d$, it is
$q^{d}\qbinom{2n}{n-d} - q^{d+1} \qbinom{2n}{n-d-1}$.
A path with depth $d$ is mapped by $\phi^d$ to a root ideal with
exactly $d$ elements in the top diagonal.
This gives the sum
\[
\sum_{o\in \mathrm{OI}(n,s)} q^{\maj(o)}=\sum_{d=0}^{s}
\left(q^d \qbinom{2n}{n-d} - q^{d+1} \qbinom{2n}{n-d-1}\right)q^{-d},
\]
which simplifies to the formula given.
\end{proof}

\begin{remark}\label{rem:rowmotion}
There is a notion of rowmotion as an action on order ideals.
Unfortunately, this action does not seem to have a nice order when restricted to
$\mathrm{OI}(n,s)$ for general $s$, see \cite{StrikerWilliams2012}. For example, for $n = 3$ and $s = 1$ we have the following orbit of length 4, implying that the action does not have the order we are looking for (which is $n = 3$). \[\ytableausetup{smalltableaux}
\ytableaushort{{*(lightgray)}{*(lightgray)}{*(lightgray)},{*(lightgray)}{*(lightgray)},{*(lightgray)}}\quad \longrightarrow \quad \ytableaushort{{*(white)}{*(white)}{*(white)},{*(white)}{*(white)},{*(white)}}\quad  \longrightarrow
\quad \ytableaushort{{*(white)}{*(white)}{*(lightgray)},{*(white)}{*(lightgray)},{*(lightgray)}}\quad
\longrightarrow \quad \ytableaushort{{*(white)}{*(lightgray)}{*(lightgray)},{*(lightgray)}{*(lightgray)},{*(lightgray)}}\quad  \longrightarrow
\quad \ytableaushort{{*(lightgray)}{*(lightgray)}{*(lightgray)},{*(lightgray)}{*(lightgray)},{*(lightgray)}}
\]
\end{remark}

\begin{lemma}\label{lem:valuesAtRootsOfUnity}
Let $\xi$ be a primitive $(2n)^\thsup$ root of unity.
For all integers $n > s \ge 0$ and $d \mid 2n$, we have
\[
Y_{n,s}\left( \xi^d \right) =
\chi(d,2) \binom{d}{\frac{d}{2}} - (-1)^d\chi(n-s-1,2n/d) \binom{d}{\frac{d(n-s-1)}{2n}}.
\]
where $\chi(a,b)$ is equal to $1$ if $b$ divides $a$ and $0$ otherwise.
\end{lemma}
\begin{proof}
The evaluation follows from the $q$-Lucas theorem, \cref{thm:q-Lucas}.
\end{proof}
In light of \eqref{eq:stretchedSYTCSP}, it would be of great interest to
explicitly describe a group action $C_n$ so that $X_{n,s}(q)$ or $Y_{n,s}(q)$ is the
corresponding polynomial in a CSP triple (which must exist due to \eqref{eq:stretchedSYTCSP}).
In \cref{sec:s=1} we give an explicit action in the case $s=1$,
which gives a new CSP triple involving the Catalan numbers.

\subsection{Narayana connection}\label{subsec:narayanaDiscussion}

We discuss an open problem regarding
the interpolation between type $A$ and type $B$ $q$-Narayana numbers.
This problem is part of a broader set of questions regarding the interplay
of cyclic sieving and characters in the symmetric group, see \cite{AlexanderssonPfannererRubeyUhlin2020x}.
We argue that the small special case discussed below is interesting in its own right.

Recall that
\[
 \sum_{D \in \DYCK(n)} q^{\maj(D)} t^{\peaks(D)} = \sum_{k=1}^n t^k \Nar[q]{n}{k},
\]
where the sum ranges over Dyck paths of size $n$.
Hence, $\sum_{k=1}^n \Nar[q]{n}{k} = X_{n,0}(q)$.
Note that the set of Dyck paths with $k$ peaks is in bijection with
the set of non-crossing set partitions of $[n]$ with $k$ blocks.

\begin{problem}[Main Narayana problem]\label{prob:narayana}
Refine the expression
\[
Y_{n,s}(q) =\qbinom{2n}{n} - q^{s+1} \qbinom{2n}{n-s-1}
\]
for all $s \geq 0$ in the same way as the \defin{$q$-Narayana numbers}
$\Nar[q]{n}{k}$ refine the case $s=0$.
\end{problem}
This problem is not really interesting
unless we impose some additional requirements.
In \cref{prob:narayana}, we are hoping to find a family of polynomials,
$\mdefin{N(s,n,k;q)} \in \setN[q]$ with some of the following properties:

\noindent
\textbf{Specializes to $\Nar[q]{n}{k}$:}
For $s=0$, we have
 \[
 N(0,n,k;q)  = \Nar[q]{n}{k}.
\]

\noindent
\textbf{Refines the $Y_{n,s}(q)$ in \eqref{eq:shiftedPaths}:}
We want that for all $s \geq 0$, we have the identity
 \[
 \sum_{k=1}^n N(s,n,k;q) = \qbinom{2n}{n} - q^{s+1}\qbinom{2n}{n-1-s}.
\]

\noindent
\textbf{Is given by some generalization of the the peak statistic:}
We hope for some statistic $\peaks_s(P)$ such that $\peaks_0(P)$
is the usual number of peaks of a Dyck path and
\begin{equation}\label{eq:peakQuest}
 \sum_{P \in \PATHS_{s}(n)} q^{\maj(P)} t^{\peaks_s(P)} = \sum_{k=1}^n t^k N(s,n,k;q).
\end{equation}
We can alternatively consider some other family of combinatorial
objects mentioned in \eqref{sec:typeABNarayanaQuest},
such as type $B$ root ideals with at most $s$ elements on the top diagonal,
or standard Young tableaux in $\SYT((n+s,n)/(s))$ with some type of generalized descents.

\noindent
\textbf{Refines $\CatB{n}$ at $s=n$:}
For $s=n$, we have a natural candidate
 \begin{equation}\label{eq:typeBNarayanaCandidate}
 N(n,n,k;q) = q^{k(k - 1)} \qbinom{n - 1}{k - 1} \qbinom{n + 1}{k} = [n+1]_q \Nar[q]{n}{k}.
\end{equation}
Note that $N(n,n,k;q)$ is not equal to $\NarB[q]{n}{k}$ that appear in \cref{sec:PrelTypeB}.
The combinatorial interpretation in this case is as follows:
 \[
   \sum_{P \in \PATHS(n)} q^{\maj(P)} t^{\mathrm{modpeaks}(P)} =
   \sum_{k=1}^n t^k
   q^{k(k - 1)} \qbinom{n - 1}{k - 1} \qbinom{n + 1}{k}
 \]
where a \defin{modified peak} is any occurrence of $01$ (north-east) in the path, plus $1$
if the path ends with a north step.

\noindent
\textbf{Palindromicity:}
The Narayana numbers have quite nice properties.
First of all,
\[
 \sum_{k=1}^n t^k \Nar{n}{k}
\]
is a \emph{palindromic} polynomial (in $t$). For example, for $n=5$, this sum is
$t + 10 t^2 + 20 t^3 + 10 t^4 + t^5.$
One would therefore hope that for fixed $s$
the sum $\sum_{k=1}^n t^k N(s,n,k;q)$ is palindromic.
The $s=n$ candidate given by $[n+1]_q \Nar[q]{n}{k}$ is also palindromic.

\noindent
\textbf{Palindromicity II:}
Each $N(s,n,k;q)$ is a palindromic polynomial (in $q$).
This is true for $\Nar[q]{n}{k}$ and the expression in \eqref{eq:typeBNarayanaCandidate}.

\noindent
\textbf{Gamma-positivity:}
The sum $\sum_{k=1}^n t^k N(s,n,k;0)$ is $\gamma(t)$-positive
(see the survey \cite{Athanasiadis2018} for the definition).
The corresponding statement seems to hold for the expression in \eqref{eq:typeBNarayanaCandidate}.
One might hope that the general expression $\sum_{k=1}^n t^k N(s,n,k;0)$
also has $\gamma(t)$-positivity.


\noindent
\textbf{Values at roots of unity and cyclic sieving:}
We require that $N(s,n,k; \xi)$ is a non-negative integer
whenever $\xi$ is an $n^\thsup$ root of unity.
This resonates well with the palindromicity properties,
and cyclic sieving for \eqref{eq:stretchedSYTCSP}.
Taking \eqref{eq:peakQuest} into account, we would like that for every $k\geq 0$,
\[
 \left( \{ P \in \PATHS_{s}(n) : \peaks_s(P)=k \},  \langle \beta_n \rangle , N(s,n,k; q) \right)
\]
is a CSP triple for some action $\beta_n$ of order $n$.
Note that such a refinement is known in the case $s=0$,
as shown in the table below. Note also that there cannot be
a cyclic group action of order $2n$ that fits together
with $N(0,n,k; q)$ in a CSP triple: for example, at $n = 4, k = 2$ this is not an integer at a primitive $2n^\thsup$ root of unity.

\begin{table}[!ht]
 \centering
\begin{tabular}{llll}
\toprule
Set & Group action & $q$-statistic & peak-statistic \\
\midrule
Dyck paths & --- & $\maj$ & $\peaks$  \\
Non-crossing partitions\textsuperscript{\textdagger} & Rotation & $\maj$ &$\blocks$ \\
\midrule
$\SYT(n^2)$ & $\prom^2$ & $\maj$ & ---  \\
Non-crossing matchings\textsuperscript{\textdagger} & Rotation & $\maj$ & --- \\
\bottomrule
\end{tabular}
\caption{
We only have the full Narayana refinement picture
for the non-crossing partition family.
That is, there is a ``peak''-statistic and a group action of order $n$ preserving the peak-statistic.
Note that promotion on Dyck paths does not preserve the number of peaks.
\textsuperscript{\textdagger}For these sets $\maj$ is computed via a bijection to paths.
}
\end{table}

\begin{example} For $n=2, s=1$, we have that $Y_{2,1}(q)=q^4+q^3+q^2+q+1$.
We want to refine this into two polynomials corresponding to $k=1,2$.
The criteria to have non-negative evaluations at roots at unity, here $-1$,
tells us that $q^3$ and $q$ must be together with at least one other term each.
By palindromicity II there are five possibilities
for $N(1,2,2;q)$: $q^4+q^3+q+1$, $q^4+q^3+q^2+q$, $q^4+q^3+q^2$, $q^4+q^3$ and $q^4$.
\end{example}

\section{Case \texorpdfstring{$s=0$}{s=0} and non-crossing matchings}\label{sec:s=0}

The goal of this section is to prove two Narayana-refinements of cyclic sieving
on non-crossing perfect matchings by considering the number of even edges and short edges.
The second result corresponds to a refinement of the CSP on $\SYT(n^2)$ under promotion,
where we refine the set by the number of cyclic descents.

\subsection{Even edge refinement}\label{sec:evenNCM}

Given a non-crossing perfect matching,
let $\mdefin{\even(M)}$ denote the number of edges $\{i,j\}$
where $i<j$ and $i$ is even. We refer to them as \defin{even edges},
and all non-even edges are called \defin{odd}. Let $\NCM(n,k)$ be the set of $M \in \NCM(n)$ such that $\even(M) = k$.

Note that for parity reasons an edge $\{i,j\}$ must have $i+j$ odd. Thus the set of non-crossing perfect matchings
on $2n$ vertices with $k$ even edges is invariant under rotation by $\rot_n$ since
\begin{itemize}
 \item any odd edge $(i,2n)$ is mapped to the even edge $(2,i+2)$;
 \item any even edge $(j,2n-1)$ is mapped to the odd edge $(1,j+2)$.
\end{itemize}

The first result is essentially just a restatement of  \cite[Thm.~7.2]{ReinerStantonWhite2004}. 

\begin{proposition}\label{prop:NCMNarayanaCSP}
For $0\leq k \leq n$, the triple
\begin{equation}
\left( \NCM(n,k), \rot_{n}, \Nar[q]{n}{k+1} \right)
\end{equation}
exhibits the cyclic sieving phenomenon.
\end{proposition}
\begin{proof}
Mapping non-crossing matchings to non-crossing partitions via the inverse
of $\NCPtoNCM$ takes matchings with $k$ even edges to partitions with $k+1$ blocks, see \cref{bij:NCPtoNCM}.
This CSP result was proven already in \cite[Thm.~7.2]{ReinerStantonWhite2004}.
\end{proof}

\subsection{Short edge refinement}

\begin{definition}
	We define \defin{promotion} $\prom_{2n}:\SYT(n^2) \to \SYT(n^2)$ as the following composition of bijections:
	\begin{equation*}
		\mdefin{\prom_{2n}} \coloneqq \SYTtoNCM^{-1} \circ \rot_{2n} \circ \ \SYTtoNCM.
	\end{equation*}
	If $T \in \SYT(n^2)$, we use the shorthand $\prom_{2n} T$ to mean $\prom_{2n}(T)$.
\end{definition}
Promotion is originally defined for Young tableaux of all shapes using the so-called
jeu-de-taquin.
The notion has been generalized to
arbitrary posets by R.~Stanley,~see~\cite{Stanley2009}.

\begin{definition}
	Let $T \in \SYT(n^2)$. Define the \defin{cyclic descent set} $\CDES(T)$ as follows.
	We have $\DES(T) =\CDES(T)\cap[1,2n-1]$ and let $2n \in \CDES(T)$ if and
	only if $1 \in \CDES(\prom_{2n} T)$.
	Denote the number of \defin{cyclic descents} of $T$ by $\mdefin{\cdes(T)} \coloneqq |\CDES(T)|$ and denote $\mdefin{\SYT_{\cdes}(n^2,k)}$ the set of $T \in \SYT(n^2)$ such that $\cdes(T)=k$.
\end{definition}
The above definition of cyclic descent set can be generalized in a straightforward
manner to all rectangular standard Young tableaux---that is, tableaux of
shape $\lambda = (a^b)$. In \cite{Huang2020}, an explicit construction is given,
where it is shown that all shapes which are not connected ribbons
admit a type of cyclic descent statistic. It follows that one can define the set $\SYT_{\cdes}(\lambda,k)$ for all such shapes $\lambda$ as well.

The set $\SYT_{\cdes}(n^2,k)$ is in bijection with a certain subset of $\DYCK(n)$ which we shall now describe.
We first recall the standard bijection $\SYTtoDYCK$ between $\SYT(n^2)$ and $\DYCK(n)$:
given a $T \in \SYT(n^2)$, let $\mdefin{\SYTtoDYCK(T)} = \sfw_1 \sfw_2 \dotsb \sfw_{2n}$ be the Dyck
path where $\sfw_i=0$ if $i$ is in the top row and $\sfw_i=1$ otherwise.

Call a Dyck path $\sfw_1 \sfw_2 \sfw_3 \dotsb \sfw_{2n}$
\defin{elevated} if $\sfw_2 \sfw_3 \dotsb \sfw_{2n-1}$ is also a Dyck path.
A Dyck path which is not elevated is called \defin{non-elevated}.
Elevated Dyck paths of size $n$ are
in natural bijection with Dyck paths of size $n-1$. The next lemma now easily follows.
\begin{lemma}\label{lem:elevatedDyck}
	Let $T \in \SYT_{\cdes}(n^2,k)$. Then $2n \in \CDES(T)$ if and only if the Dyck path $\SYTtoDYCK(T)$ is elevated.
\end{lemma}
It follows from \cref{lem:elevatedDyck} that the restriction of $\SYTtoDYCK$ to $\SYT_{\cdes}(n^2,k)$ is a
bijection to the set of $D \in \DYCK(n)$ such that $D$ either is non-elevated and has $k$ peaks or is
elevated and has $k-1$ peaks.
Hence, we get (using an argument by T.~Do{\v{s}}li{\'{c}}~\cite[Prop.~2.1]{Dosilic2010}),

\begin{align}
|\SYT_{\cdes}(n^2,k)| &= \Nar{n}{k} - \Nar{n-1}{k} + \Nar{n-1}{k-1} \notag \\
			  &= \frac{2}{n+1}\binom{n+1}{k}\binom{n-2}{k-2}. \label{eq:cDesRefinement}
\end{align}
These numbers are a shifted variant of the OEIS entry \oeis{A108838}.
Define the following $q$-analog of these numbers.
For any two natural numbers $n$ and $k$, let
\begin{equation}
	\mdefin{\qSYT[q]{n}{k}} \coloneqq \sum_{T \in \SYT_{\cdes}(n^2,k)} q^{\maj(T)-n} = \sum_D q^{\maj(D)}
\end{equation}
where the second sum is taken over all $D \in \DYCK(n)$ that are either
non-elevated with $k$ peaks or elevated with $k-1$ peaks. For integers $k,n \geq 1$ we claim that
\begin{equation}
	\qSYT[q]{n}{k} = \Nar[q]{n}{k}-q^{k-1} \Nar[q]{n-1}{k}+q^{k-2} \Nar[q]{n-1}{k-1}. \label{eq:g-polynomial}
\end{equation}
To see this, consider the restriction
of $\SYTtoDYCK$ to $\SYT_{\cdes}(n^2,k)$. If $D$ is an elevated Dyck path of
size $n$ with $k$ peaks and $D'$ is the corresponding Dyck path
of size $n-1$, then $\maj(D)-\maj(D')=k-1$, as each of the $k-1$
valleys contribute one less to the major index in $D'$ compared to in $D$.

The polynomials $\qSYT[q]{n}{k}$ refine the $q$-Catalan numbers,
which is easily seen by comparing their definition with \eqref{eq:qCatalanDef}.
\begin{proposition}
For all integers $n$,
\[
	\sum_k \qSYT[q]{n}{k}=\Cat[q]{n}.
\]
\end{proposition}
It is easy to see that $\qSYT[q]{0}{0}=\qSYT[q]{1}{1}=1$ and $\qSYT[q]{n}{k}=0$ for all other pairs of natural
numbers $n$, $k$ such that either $n \leq 1$ or $k \leq 1$.
For larger $n$ and $k$, we have the following closed form for $\qSYT[q]{n}{k}$.
\begin{lemma}
For all integers $k, n \geq 2$,
\begin{equation}\label{eq:maj-polSYT}
	\qSYT[q]{n}{k} = \frac{q^{k(k-2)}(1+q^n)}{[n+1]_q} \qbinom{n+1}{k} \qbinom{n-2}{k-2}.
\end{equation}
\end{lemma}
\begin{proof}
We may restrict ourselves to the case when $n \geq k$ as both
sides of \eqref{eq:maj-polSYT} are identically zero otherwise.
We write $\qSYT[q]{n}{k}$ using the expression in \eqref{eq:g-polynomial}
and expand the $q$-Narayana numbers to obtain
\begin{multline*}
	\frac{q^{k(k-1)}}{[n]_q}\qbinom{n}{k}\qbinom{n}{k-1}
	- q^{k-1}\frac{q^{k(k-1)}}{[n-1]_q}\qbinom{n-1}{k}\qbinom{n-1}{k-1} \\
	+ q^{k-2}\frac{q^{(k-1)(k-2)}}{[n-1]_q}\qbinom{n-1}{k-1}\qbinom{n-1}{k-2}\\[1em]
	= \frac{q^{k(k-2)}}{[n+1]_q}\qbinom{n+1}{k}\qbinom{n-2}{k-2}
	\left( q^k \frac{[n-1]_q}{[k-1]_q}-q^{2k-1}\frac{[n-k+1]_q[n-k]_q}{[n]_q[k-1]_q}+\frac{[k]_q}{[n]_q}\right).
\end{multline*}

The expression in the parentheses is then rewritten as
\[
\frac{q^k[n]_q[n-1]_q-q^{2k-1}[n-k+1]_q[n-k]_q+[k]_q[k-1]_q}{[n]_q[k-1]_q}.
\]
We must now show that this is equal to $1+q^n$ or, equivalently, that the following identity holds:
\begin{equation}\label{eq: q-integer}
	q^k[n]_q[n-1]_q-q^{2k-1}[n-k+1]_q[n-k]_q+[k]_q[k-1]_q=(1+q^n)[n]_q[k-1]_q.
\end{equation}
If $n=k$, the identity is clearly true, so we may assume that $n>k$.
By using $[j]_q = (1-q^j)/(1-q)$ for integers $j \geq 1$, we get the equivalent equation
\begin{multline*}
\frac{q^k(1-q^n)(1-q^{n-1})}{(1-q)^2}  +
\frac{(1-q^{k})(1-q^{k-1})}{(1-q)^2} \\
= \frac{(1+q^n)(1-q^n)(1-q^{k-1})}{(1-q)^2}+
\frac{q^{2k-1}(1-q^{n-k+1})(1-q^{n-k})}{(1-q)^2}.
\end{multline*}
Clearing denominators and expanding the products gives
\begin{multline*}
  \left(q^{2n+k-1} + q^k - q^{n+k-1} - q^{n+k} \right)+
 \left(q^{2k - 1} - q^{k-1} - q^k + 1 \right) = \\
   \left( q^{k+2 n-1}-q^{k-1}-q^{2 n}+1 \right) +
 \left(q^{2 n}-q^{k+n-1}-q^{k+n}+q^{2 k-1}\right)
\end{multline*}
which evidently holds for all integers $n,k$.
\end{proof}
The edge $xy$ in a non-crossing perfect matching is said to be \defin{short} if either $x=i$ and $y=i+1$ for some $i$ or
if $x=1$ and $y=2n$. If $M \in \NCM(n)$, then we denote $\mdefin{\sml}(M)$
its number of short edges and $\mdefin{\NCM_{\sh}(n,k)}$ the set of $M \in \NCM(n)$ such that $\sml(M)=k$.
The set $\SYT_{\cdes}(n^2,k)$ is in a natural bijection with $\NCM_{\sh}(n,k)$.
To see this, we use the standard bijection $\SYTtoNCM$ between $\SYT(n^2)$ and $\NCM(n)$,
see \cref{bij:SYTtoNCM} in  \Cref{sec:NCMandBW}.
\[
{
	\ytableausetup{boxsize=1.2em}
	\ytableaushort{12568,3479{10}} \quad \xrightarrow{\quad \SYTtoNCM \quad} \quad \matching{10}{1/4,2/3,5/{10},6/7,8/9}{}
}
\]

It follows from our definition of promotion that $\SYTtoNCM$
is an equivariant bijection in the sense that
\begin{equation}\label{eq: SYTbij commute}
\SYTtoNCM(\prom_{2n} T) = \rot_{2n} (\SYTtoNCM(T)).
\end{equation}

From the definition of $\SYTtoNCM$ and \eqref{eq: SYTbij commute},
one can prove the following lemma.
\begin{lemma}
Let $T \in \SYT(n^2)$.
Then $x \in \CDES(T)$ if and only if $xy$,
where $x <y $, is a short edge in $\SYTtoNCM(T)$.
\end{lemma}


\begin{theorem}\label{thm:PMrefinedCSP}
Let $n, k$ be natural numbers. The triple
\[
(\NCM_{\sh}(n,k),\langle \rot_{2n} \rangle, \qSYT[q]{n}{k})
\]
exhibits the cyclic sieving phenomenon.
\end{theorem}
\begin{proof} Let $\xi$ be a primitive $(2n)^\thsup$ root of unity.
Write $k=k_1 \order(\xi^d)+k_0$ for the unique natural
numbers $k_1$ and $k_0$ such that $0 \leq k_0 < \order(\xi^d)$. Then, by dividing into cases and applying \cref{thm:q-Lucas} (the $q$-Lucas theorem) twice, we get
\[
\qSYT[\xi^d]{n}{k} = \begin{cases}
	\frac{2}{n+1}\binom{n+1}{k}\binom{n-2}{k-2} & \text{if }d=2n,\\[1em]
	2\binom{n/\order(\xi^d)}{k_1}\binom{n/\order(\xi^d)-1}{k_1-1} & \text{if }\order(\xi^d) \mid n \text{ and } k_0=0,\\[1em]
	\frac{2n}{n+1}\binom{(n+1)/2}{k_1}\binom{(n-3)/2}{k_1-1} & \text{if } \order(\xi^d)=2, \ n \text{ odd and } 2 \mid k,\\[1em]
	0 & \text{otherwise.}
\end{cases}
\]

We prove that these evaluations agree with the number of fixed points
in $\NCM_{\sh}(n,k)$ under $\rot_n^d$ on a case-by-case basis.

\noindent
\textbf{Case $d=2n$}: Trivial.

\noindent
\textbf{Case $\order(\xi^d) \mid n$ and $k_0=0$}: By using \cref{bij:BwToNCM},
we see that such rotationally symmetric perfect matchings are in bijection
with the set $\BW^{k_1}(n/\order(\xi^d))$.
To see that this set has the desired cardinality,
we equate the two expressions in \eqref{eq:maj polynomial bin words}
and \eqref{eq:typeBSecondNarayanaPol} and then take $q=1$.

\noindent
\textbf{Case $\order(\xi^d)=2$, $n$ odd and $2 \mid k$}:
It is easy to check that the assertion holds in the case $n=3$ and $k=2$. It thus remains to show the assertion for $n>3$. Such a non-crossing perfect matching must have a diagonal
(an edge that connects two vertices $i$ and $i+n \pmod{2n}$) that divides
the matching into two halves. The diagonal can be chosen in $n$ ways.
The matching is now determined uniquely by one of its two halves.
To choose one half, we choose a non-crossing matching on $(n-1)/2$
vertices with $k/2$ short edges, not including a potential short edge between the vertices closest to the diagonal.
Such a matching is either i) an element of $\NCM_{\sh}((n-1)/2, k/2)$ which does not have
an edge between the two vertices that lie closest to the diagonal or ii) an
element of $\NCM_{\sh}((n-1)/2, k/2+1)$ which has a short edge between the
two vertices that lie closest to the diagonal.

Let us note that, in general, the fraction of elements in $\NCM_{\sh}(n,k)$ that have a short edge
adjacent to a given side is equal to $k/2n$. This is easily seen by considering
rotations of such a non-crossing perfect matching. Hence, in our case the number of
matchings fixed by $\rot_{2n}^d$ is equal to
\[
n\left(
\frac{n-1-k/2}{n-1} \left|\NCM_{\sh}\left(\frac{n-1}{2},\frac{k}{2}\right)\right|
+
\frac{k/2+1}{n-1}   \left|\NCM_{\sh}\left(\frac{n-1}{2},\frac{k}{2}+1\right)\right|
\right)
\]
Substituting the values from \eqref{eq:cDesRefinement} (recall that $\left|\NCM_{\sh}\left(a,b\right)\right|=\left|\SYT_{\cdes}\left(a^2,b\right)\right|$), it remains to show that this expression is identical to the one given by $\qSYT[-1]{n}{k}$. This can now be verified with a computer algebra system, such as Sage~\cite{Sage}.

\noindent
\textbf{The remaining cases}: We need to show that, in all the remaining cases,
there are no rotationally symmetric non-crossing perfect matchings.
Suppose first that $\order(\xi^d)=2$, $n$ is odd and $2 \nmid k$.
It is clear that such a non-crossing perfect matching must have a diagonal
dividing the matching into two halves. The two halves are identical up to
a rotation of $\pi$ radians and so, in particular, they must have the
same number of short edges. In other words, the number of short edges
must be even, contradicting $2 \nmid k$. Suppose next that $\order(\xi^d) \mid n$
and $k_0 \neq 0$. Such a matching is completely determined by how
the vertices $1,2,\dotsc, d$ are paired up. It follows that the number of
short edges must be a multiple of $2n/d$, contradicting $k_0 \neq 0$.

Suppose lastly that $\order(\xi^d) \mid 2n$ but $\order(\xi^d) \nmid n$.
To analyze this case, we first prove the following. \\

\textbf{Claim:}
If $\order(\xi^d) \mid 2n$ and $\order(\xi^d) \nmid n$, then $d$ is odd.
\begin{proof}[Proof of Claim]
The hypothesis implies that the number of $2$'s in the prime factorization
of $2n$ is equal to the number of $2$'s in the prime factorization of $o(\xi^d)$.
Hence, the number $2n/o(\xi^d)$ is odd. Combining this with
the fact that $o(\xi^d)=2n/\gcd(2n,d)$ yields that $\gcd(2n,d)$ is odd.
But this implies that $d$ is odd, so we are done.
\end{proof}
We now use the claim and note that there cannot be any
non-crossing perfect matchings that are fixed under rotation by an odd
number of steps, except in the case when $\order(\xi^d)=2$
and $n$ is odd, i.e.~when the matching has a diagonal.
This exhausts all possibilities and thus the proof is complete.
\end{proof}

\Cref{thm:PMrefinedCSP} can be stated in an alternative way as follows.
Since $\SYTtoNCM$ maps cyclic descents to short edges, we see that $\SYT_{\cdes}(\lambda,k)$ is closed
under promotion for all rectangular $\lambda$.
Recall that $(\SYT(\lambda), \langle \prom \rangle, \Cat[q]{n})$ exhibits the
cyclic sieving phenomenon, for rectangular $\lambda$.
In the case when $\lambda=(n,n)$, we have the following
refinement with regards to the number of cyclic descents.
\begin{corollary}\label{cor:SYTrefinedCSP}
	Let $n, k$ be natural numbers. The triple
	\[
	(\SYT_{\cdes}(n^2,k),\langle \prom_{2n} \rangle, \qSYT[q]{n}{k})
	\]
	exhibits the cyclic sieving phenomenon.
\end{corollary}
A related result is alluded to by C. Ahlbach, B. Rhoades and J. Swanson
in the presentation slides \cite{Swanson2018s}.
They claim to have proven a refinement of the cyclic sieving phenomenon on
standard Young tableaux with the group action being promotion in the Catalan case.
It is not clear from the slides if they refine by cyclic descents
or in some other way. Therefore, we cannot tell if their result
is identical to \cref{thm:newCatalanCSP} or not.

It follows from \cite[Lemma 3.3]{Rhoades2010} that the number of cyclic descents remains fixed under promotion
of rectangular standard Young tableaux.
Experiments suggests that \cref{cor:SYTrefinedCSP}
generalizes to all rectangular standard Young tableaux.
This would be a refinement of the famous CSP result on rectangular tableaux, see \cite[Theorem 1.3]{Rhoades2010}.
More precisely, we denote $\mdefin{f^\lambda_k(q)} \coloneqq  \sum_T q^{\maj(T)}$ where
the sum is taken over all standard Young tableaux of shape $\lambda$ with exactly $k$ cyclic descents.
\begin{conjecture}
Let $n,m, k$ be natural numbers and put $\lambda = (n^m)$. The triple
\[
\left(\SYT_{\cdes}(\lambda,k), \langle\prom_{nm} \rangle, q^{-\kappa(\lambda)} f^{\lambda}_k(q) \right)
\]
exhibits the cyclic sieving phenomenon.
Here, $\mdefin{\kappa(\lambda)} \coloneqq \sum_i (i-1)\lambda_i$.
\end{conjecture}

\section{Case \texorpdfstring{$s=1$}{s=1} and non-crossing (1,2)-configurations}\label{sec:s=1}

For $s=1$, there is a nice Catalan family, given by non-crossing $(1,2)$-configurations
described in \cite[Family 60]{StanleyCatalan}.
In the first subsection, we introduce a twisted rotation action on such configurations,
and prove a new instance of Catalan CSP. In the second subsection,
we refine a CSP result of Thiel, where the group action is given by rotation.

\subsection{A new Catalan CSP under twisted rotation}

A \defin{non-crossing $(1,2)$-configuration} of size $n$ is
constructed by placing vertices $1,\dotsc,n-1$ around a circle,
and then drawing some non-intersecting edges between the vertices.
Here, we allow vertices to have a loop, which is counted as an edge.
There are $\Cat{n}$ elements in this family.
Let $\NCC(n)$ be the set of such objects of size $n$,
and let $\NCC(n,k)$ be the subset of those with $k-1$ edges, loops included.
See \cref{sec:catalanObjectsAppendix} for a figure when $n=3$.

\begin{bijection}[Laser construction]\label{def:laserBij}
	See Figure \ref{fig:laser} for an example. Let $P \in \DP(n)$.
	Define the non-crossing $(1,2)$-configuration $\mdefin{\laser(P)}$ as follows.
	First, number the east-steps with $1,2,\dotsc, n-1$.
	Secondly, if there is a valley at $(i,j)$, draw a line (a laser)
	from $(i,j)$ to $(i + \Delta, j + \Delta)$, where $\Delta$ is the
	smallest positive integer such that $(i + \Delta, j + \Delta)$ lies on $P$.
	Now, consider an east-step ending in $(i_1,j_1)$ on $P$.
	If there is a laser drawn from $(i_1,j_1)$, then let $(i_2, j_2)$
	be the vertex of $P$ where this laser ends. Then there is an edge
	between $j_1$ and $j_2-1$ in $\laser(P)$ (this can be a loop).
	The remaining vertices in $\laser(P)$
	will be unmarked, that is, unpaired and without a loop.
\end{bijection}

\begin{proposition}\label{prop:Lazer bijection}
	The map $\laser$ is a bijection $\DP(n) \to \NCC(n)$.
\end{proposition}
A proof of \cref{prop:Lazer bijection} can essentially be found in \cite[Prop.~6.5]{Bodnar2019}.
M.~Bodnar studies so called \defin{$n+1,n$-Dyck paths} and shows that these are
in bijection with $\NCC(n)$. It is not hard, however, to see that
the set of $n+1,n$-Dyck paths is in bijection
with $\DYCK(n)$ by removing the first north-step.

Note that there is a natural correspondence between Dyck paths
of size $n$ and paths of size $n-1$ that stay weakly above the diagonal $y=x-1$.
If $P=\sfw_1 \sfw_2 \dotsb \sfw_{2n-1} \sfw_{2n}$ is a Dyck path of
size $n$, then let $P'=\sfw_2 \dotsb \sfw_{2n-1}\in \PATHS_1(n-1)$.
Furthermore, $P$ and $P'$ have the same number of valleys.

\begin{figure}[!ht]
\[
	\begin{tikzpicture}[scale=0.5,baseline=(current bounding box.center)]
	\draw (0,0) grid (6,6);
	\draw (0,0) [Dyck path={1,1,0,1,0,0,1,1,0,1,0,0}];
	\end{tikzpicture}
	\quad \longleftrightarrow \quad
	\matching{5}{3/5}{1,4}
\]
\caption{An example of the bijection $\laser$ in \cref{def:laserBij}.}\label{fig:laser}
\end{figure}

\begin{lemma}\label{lem:nccNarayana}
 We have that $|\NCC(n,k)| = \Nar{n}{k}$, the Narayana numbers.
\end{lemma}
\begin{proof}
The bijection $\laser$ maps Dyck paths with $k-1$ valleys to
non-crossing $(1,2)$-configurations with $k-1$ edges.
It remains to note that a Dyck path with $k-1$ valleys has $k$ peaks.
\end{proof}

\begin{remark}
Recall that the \defin{Motzkin numbers} $M_i$ count the number of ways to draw
non-intersecting chords on $i$ vertices arranged around a circle, see \oeis{A001006}
in the OEIS.
The set
\[
 \{ C \in  \NCC(n+1) : \mathrm{loops}(C) = l \}
\]
has cardinality $\binom{n}{l}M_{n-l}$ sice we can first choose the $l$ vertices
which have loops, and then proceed by choosing one of the $M_{n-l}$
possible arrangements of non-intersecting chords
on the remaining $n-l$ vertices.
\end{remark}

Let $\rot_n$ denote rotation by one step, acting on $\NCC(n+1)$.
Furthermore, let $\mdefin{\flip}$ denote the the action
of removing the mark on vertex $1$ if it is marked,
and marking it if it is unmarked. It does not do
anything if $1$ is connected to an edge.
We refer to this as a \defin{flip}.

Let the \defin{twist action} be defined as $\mdefin{\twist_{2n}} \coloneqq \rot_n \circ \flip$.
It is straightforward to see that $\twist_{2n}$ generates a cyclic group of order $2n$
acting on $\NCC(n+1)$.
Alternatively, we can act by $(\rot_n \circ \flip)^{n-1}$,
which closely resembles \emph{promotion}.
Recall that promotion on SYT may be defined as a sequence of swaps,
for $i=1,2,\dotsc,n-1$, where swap $i$ interchanges the
labels $i$ and $i+1$ if possible.
\begin{figure}[!ht]
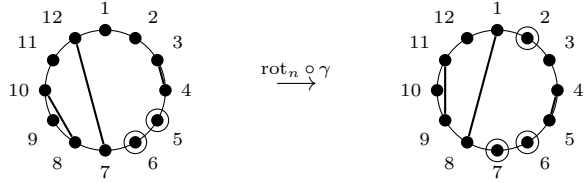

	\centering
	\[
	\matching{12}{3/4,7/12, 8/10}{5,6}
	\qquad
	\stackrel{\rot_n \circ \flip}{\longrightarrow}
	\qquad
	\matching{12}{4/5,  8/1, 9/11}{2,6,7}
	\]
	\caption{The result of $\rot_n \circ \flip$ on an element in $\NCC(13)$.}
\end{figure}

\begin{theorem}[A new cyclic sieving on Catalan objects]\label{thm:newCatalanCSP}
The triple
 \[
  \left( \NCC(n+1), \langle \twist_{2n} \rangle, \qbinom{2n}{n}  -  q^{2}\qbinom{2n}{n-2} \right)
 \]
exhibits the cyclic sieving phenomenon. Note that
\[
 \qbinom{2n}{n}  -  q^{2}\qbinom{2n}{n-2} = \frac{[2]_q}{[n+2]_q}\qbinom{2n+1}{n}.
\]
\end{theorem}
\begin{proof}
We compute the number of fixed points under $\twist_{2n}^m$, where we may without loss of generality assume $m \mid 2n$.
There are two cases to consider, $m$ odd and $m$ even.
In the first, we must, according to \cref{lem:valuesAtRootsOfUnity}, show that
the number of fixed points under $\twist_{2n}^m$ is
\[
\begin{cases}
 \binom{m}{(m-1)/2} & \text{ if $m = n/2$ is odd,} \\
 0 & \text{ otherwise.}
 \end{cases}
\]
For the first expression we reason as follows. Since $m$ is odd, any fixed point for such $m$ must consist of a diagonal (an edge from $i$ to $i+m$) and two rotationally symmetrical halves, both consisting of $m-1$ vertices. In such a non-crossing configuration, no vertex can be isolated. To see this, note that if vertex $j$ is isolated, then so are $j+km \pmod{n}$, $k \in \mathbb{Z}$, but the $j+km \pmod{n}$ would need to be both marked and unmarked, a contradiction. Thus, all vertices in the non-crossing configuration are incident to an edge and it is in fact a non-crossing matching. The diagonal can be chosen in $m$ ways and a non-crossing matching on one of the two halves can be chosen
in $\Cat{(m-1)/2}$ ways, so there are $m\Cat{(m-1)/2} = \binom{m}{(m-1)/2}$ fixed points.

In the second case above, if $n = m$ then at least one vertex has to be isolated since $m$ is odd,
which implies there can be no fixed points. For $n/m > 2$, we use a ``parity'' argument.
Since any isolated vertices among $S = \{n-m+2, n-m+3, \dots, n, 1\}$
change from unmarked to marked and vice versa under $\twist_{2n}^m$, the number of isolated vertices has to be even.
Since $m$ is odd, this implies there must be an odd number of edges from $S$ to $[n] \setminus S$ in a fixed point. However, note that $S$ and the edges out of
$S$ completely determine the configuration. Hence the edges must have their other endpoints in the two neighboring intervals of length $m$. But this violates being rotationally symmetric under rotations of $m$ steps since the number of edges is odd.

In the case $2 \mid m$, according to \cref{lem:valuesAtRootsOfUnity} we must show
that the number of fixed points under $\twist_{2n}^m$ is equal to
\[
\begin{cases}
 \binom{2n}{n} - \binom{2n}{n+2} & \text{ if } m=2n, \\
 \Cat{n/2} & \text{ if $m=n$ is even},\\
 \binom{m}{m/2} & \text{ otherwise.}
 \end{cases}
\]
Counting fixpoints for the first case is trivial. For the second expression,
note that $\twist_{2n}^{n}$ is simply the action
of flipping the markings on all vertices.
A fixed point can therefore not have any isolated vertices.
What remains are non-crossing matchings, which there are $\Cat{n/2}$ many of.

It remains to prove the third expression. Let $m = 2d$.

\textbf{Case $n$ and $n/d$ are even.} In this case, the only possible
invariant configurations are non-crossing matchings that are rotationally symmetrical
when rotating $2d$ steps.
Recall from \cref{bij:BwToNCM} that such matchings are in
bijection with $\BW(2d,d)$ and this set clearly has cardinality $\binom{2d}{d}$.

\textbf{Case $n/d$ is odd.}
Here we can have fixed points under $\twist_{2n}^{m}$ with unpaired vertices,
for examples see \cref{F:nd udda}.
The orbit of a vertex $j$ under the operation $\twist_{2n}^{m}$
is $\{j+dk \pmod {n}\}_{k\in \mathbb Z}$ and if $1\le j\le d$ is an unpaired,
unmarked (that is, without a loop) vertex, the vertices $\{j+2dk \pmod {n}\}_{0\le k< n/2d}$
must be unmarked whereas $\{j+2dk \pmod {n}\}_{n/2d< k< n/d }$ will be marked.
Note that in the latter case $j+2dk \pmod {n}\equiv j+d+2dr$ for $r=k-(n/d+1)/2$.
Thus it suffices to understand the vertices from 1 to
$d$. We claim that for every $0\le i\le \lfloor d/2\rfloor$ we get a valid
fixed point by choosing $i$ left vertices
and $i$ right vertices and matching them in a non-crossing manner as in the previous case.
Then we can choose to put a loop at any subset of the remaining $d-2i$ unpaired vertices.
Every fixed point is now constructed exactly once.
This gives a total of $\sum_{i=0}^{\lfloor d/2\rfloor} \binom{d}{i}\binom{d-i}{i}2^{d-2i}$ fixed points.
Finally we need to prove that this sum is equal to $\binom{2d}d$.
We will use a bijection to all possible subsets $A$ of size $d$ from two rows
with numbers 1 to $d$, the numbers in the top row being blue and the bottom row red.
For a given $i$ we choose $i$ numbers and let both the red and the blue
belong to $A$ and then $i$ numbers such that neither blue nor red belong to $A$.
Finally the term $2^{d-2i}$ corresponds to choosing any subset of the
remaining $d-2i$ numbers such that the red numbers in that subset belong
to $A$ and the blue in the complement are in $A$.


\begin{figure}[!ht]
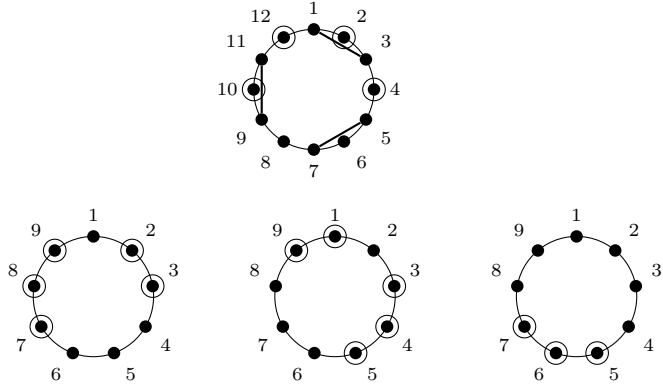

	\centering
	\[
	\matching{12}{1/3, 5/7, 9/11}{2,4,10,12} \qquad
	\]
	\centering
	\[
	\matching{9}{}{2,3,7,8,9}\qquad
	\matching{9}{}{1,3,4,5,9}\qquad
	\matching{9}{}{5,6,7}
	\]
\caption{A fixed point under $\twist_{2n}^m$, $n=12$, $m=8$ and $2n/m=3$, and below an
orbit under $\twist_{2n}^2$ of size $3$.}\label{F:nd udda}
\end{figure}
\end{proof}

\begin{problem}
It would be nice to refine the CSP triple in \cref{thm:newCatalanCSP}
to the hypothetical Narayana case discussed in \cref{sec:typeABNarayanaQuest}.
That is, we want the following equality to hold:
\[
 \qbinom{2n}{n}  -  q^{2}\qbinom{2n}{n-2} = \sum_{k=1}^{n} N(1,n,k;q)
\]
where $N(1,n,k;1)$ is the number of NCC on $n$ vertices with $(k-1)$ edges.
This is a natural consideration, as indicated by \cref{lem:nccNarayana}.
\end{problem}

\subsection{A refinement of Thiel's result}\label{sec:nccRefinement}

Recall that $\rot_n$ acts on non-crossing $(1,2)$-configurations of
size $(n+1)$ via a $2\pi/n$-rotation. M.~Thiel proved the following.
\begin{proposition}[See \cite{Thiel2017}.]\label{prop:thiel}
Let $n \in \setN$. The triple
\[
\left( \NCC(n+1), \langle \rot_n \rangle, \Cat[q]{n+1} \right)
\]
exhibits the cyclic sieving phenomenon.
\end{proposition}

Denote $\mdefin{\NCC(n+1, e, l)}$ the number of non-crossing $(1,2)$-configurations
with $n$ vertices, $e$ proper edges and $l$ loops.
We determine an element in $\NCC(n+1,e,l)$ by first choosing the $2e$ vertices that
are incident to some proper edge in $\binom{n}{2e}$ ways, then choosing a
non-crossing matching among these $2e$ vertices in $\Cat{e}$ ways and finally
choosing $l$ of the remaining $n-2e$ vertices to be loops. Hence,
\begin{equation}\label{eq:nccNELCount}
| \NCC(n+1, e, l) | = \binom{n}{2e} \Cat{e} \binom{n-2e}{l}.
\end{equation}
For any $e, l \in \setN$, define the following $q$-analog of the above the expression:
\begin{align}
\mdefin{\qNCC[q]{n}{e}{l}} &\coloneqq q^{e(e+1)+(n+1)l}\qbinom{n}{2e}\Cat[q]{e}\qbinom{n-2e}{l} \\
&= q^{e(e+1)+(n+1)l} \frac{1}{[e+1]_q}\qbinom{n}{e,e,l,n-2e-l}.
\end{align}

\begin{example}
	Consider the case with $n=4$ and $k=2$ proper edges and loops. We have
	\begin{align*}
		\qNCC[q]{4}{2}{0} &= q^6(1+q^2)\\
		\qNCC[q]{4}{1}{1} &= q^7(1+2q+3q^2+3q^3+2q^4+q^5)\\
		\qNCC[q]{4}{0}{2} &= q^{10}(1+q+2q^2+q^3+q^4)
	\end{align*}
	It is easily verified that these polynomials refine $\Nar[q]{5}{3}$. That is,
	\begin{align*}
	\Nar[q]{5}{3} &=q^6(1+q+3q^2+3q^3+4q^4+3q^5+3q^6+q^7+q^8) \\
	&=\qNCC[q]{4}{2}{0}+\qNCC[q]{4}{1}{1}+\qNCC[q]{4}{0}{2}.
\end{align*}
\end{example}

\begin{lemma}\label{lem:refinementOfqNarayana}
 For every $k\geq 0$ we have the identity
\[
\Nar[q]{n+1}{k+1} = \sum_{e+l=k} \qNCC[q]{n}{e}{l}.
\]
\end{lemma}
\begin{proof}
 Unraveling the definitions, it suffices to show that
  \begin{equation*}
  \frac{q^{k(k+1)}}{[n+1]_q} \qbinom{n+1}{k+1}\qbinom{n+1}{k}   =
  \sum_{e+l=k} q^{e(e+1)+(n+1)l} \qbinom{n}{2e} \qbinom{n-2e}{l}  \frac{1}{[e+1]_q} \qbinom{2e}{e}.
 \end{equation*}
 Expanding the $q$-binomials
 gives
 \begin{align*}
  &\frac{q^{k(k+1)}}{[n+1]} \frac{[n+1]!}{[k+1]![n-k]!}\frac{[n+1]!}{[k]![n-k+1]!} \\
  &= \sum_{0 \leq e \leq k }\frac{ q^{e(e+1)+(n+1)(k-e)} [n]!}{[2e]![n-2e]!} \frac{[n-2e]!}{[k-e]![n-k-e]!}
  \frac{[2e]!}{[e+1]![e]!},
 \end{align*}
 where we have omitted the $q$-subscripts for brevity.
 We start doing cancellations,
\begin{equation*}
q^{k(k+1)} \frac{[n]!}{[k+1]![n-k]!}\frac{[n+1]!}{[k]![n-k+1]!}   =
  [n]! \sum_{0 \leq e \leq k } \frac{1}{[e+1]} \frac{q^{e(e+1)+(n+1)(k-e)} }{[k-e]![n-k-e]![e]![e]!},
 \end{equation*}
and additional cancellations and some rewriting gives
 \begin{equation*}
  q^{k(k+1)} \frac{[n+1]!}{[k]![k+1]![n-k]![n-k+1]!}   =
  \sum_{0 \leq e \leq k } \frac{q^{e(e+1)+(n+1)(k-e)} }{[k-e]![n-k-e]![e]![e+1]!}.
 \end{equation*}
Further rewriting now gives
 \begin{align*}
  q^{k(k+1)} \frac{[n+1]!}{[n-k]![k+1]!}  =
  \sum_{0 \leq e \leq k } q^{e(e+1)+(n+1)(k-e)}  \frac{[k]!}{[k-e]![e]!}\frac{[n-k+1]!}{[n-k-e]![e+1]!}.
 \end{align*}
 Thus, the identity we wish to prove is equivalent to showing that
  \begin{align*}
 \qbinom{n+1}{k+1}  =
  \sum_{0 \leq e \leq k } q^{(k- e)(n - e - k)}  \qbinom{k}{k-e} \qbinom{n-k+1}{e+1}.
 \end{align*}
However, this follows from the $q$-Vandermonde identity (\cref{thm:q-vandermonde})
by substituting $a = n+1-k$, $b=k$, $c=k+1$ and $j=k-e$.
\end{proof}

It is clear that one can restrict the action of $\rot_n$ to $\NCC(n+1,e,l)$.
The following result is a refinement of \cref{prop:thiel}.
\begin{theorem}\label{thm:thielRefinement}
Let $n,e,l \in \setN$. The triple
\[
(\NCC(n+1,e,l), \langle \rot_n \rangle, \qNCC[q]{n}{e}{l})
\]
exhibits the cyclic sieving phenomenon.
\end{theorem}
\begin{proof}
Let $\xi$ be a primitive $n^\thsup$ root of unity and let $d\mid n$.
Write $e=e_1(n/d)+e_0$ and $l=l_1(n/d)+l_0$ for the unique natural numbers $e_1, e_0, l_1, l_0$ such
that $0 \leq e_0 <n/d$ and $0 \leq l_0 < n/d$. Using \cref{thm:q-Lucas} (the $q$-Lucas theorem) repeatedly, we get
\[
\qNCC[\xi^d]{n}{e}{l} =
\begin{cases}
\binom{n}{2e} \Cat{e} \binom{n-2e}{l} & \text{if }d=n,\\[1em]
\binom{d}{2e_1}\binom{2e_1}{e_1}\binom{d-2e_1}{l_1} & \text{if } e_0=0 \text{ and } l_0=0,\\[1em]
\binom{d}{e}\binom{e}{e_1}\binom{d-e}{l_1} & \text{if } d=n/2, \ e_0=1 \text{ and } l_0=0, \\[1em]
0 & \text{otherwise.}
\end{cases}
\]
We prove that these evaluations agree with the number of
fixed points in $\NCC(n+1,e,l)$ under $\rot_n^d$ on a case-by-case basis.

\noindent
\textbf{Case $d=n$}: Trivial.

	\noindent
	\textbf{Case $e_0=0$ and $l_0=0$}:
	A $(1,2)$-configuration that is fixed by $\rot_n^d$ is completely
	determined by its first $d$ vertices. Among these $d$ vertices, there must $2e(d/n)=2e_1$ vertices
	that are incident to an edge and $l(d/n)=l_1$ loops.
	There are $\binom{d}{2e_1}$ ways to choose $2e_1$ from the first $d$ vertices.
	The number of ways to arrange these edges in an admissible way is equal to
	the number of perfect matchings that are invariant when rotating $2e_1$ steps.
	By \cref{bij:BwToNCM}, we know that there are $\binom{2e_1}{e_1}$ such matchings.
	Lastly, choose $l_1$ loops among the remaining $d-2e_1$ vertices
	in $\binom{d-2e_1}{l_1}$ ways.
	These choices are all independent and the desired result follows.

	\noindent
	\textbf{Case $d=n/2$, $e_0=1$ and $l_0=0$}:
	Such a $(1,2)$-configuration must have a diagonal (an edge from $i$ to $i+d$) that
	splits the $(1,2)$-configuration into two halves.
	The diagonal can be chosen in $d$ ways. The $(1,2)$-configuration is now
	determined uniquely by one of its halves. Such a half must have $d-1$ vertices
	with $(e-1)/2=e_1$ edges and $l_1$ loops. Choose the $2e_1$ vertices that are
	incident to an edge from the $d-1$ vertices in $\binom{d-1}{2e_1}$ ways.
	The number of the ways to arrange these edges in an admissible way is
	equal to the number of non-crossing perfect matchings on $2e_1$ vertices,
	namely $\Cat{e_1}$. Finally, choose $l_1$ loops from the remaining $d-e$ vertices
	in $\binom{d-e}{l_1}$ ways. Since these choices are independent,
	the number of fixed points is given by
	\[
	 d \binom{d-1}{2e_1}\Cat{e_1}\binom{d-e}{l_1} = \binom{d}{e}\binom{e}{e_1}\binom{d-e}{l_1}
	\]
	where equality follows from some simple manipulations of binomial coefficients.

	\noindent
	\textbf{The remaining cases}:
	Suppose that $P \in \NCC(n+1, e, l)$ is invariant under $\rot_n^d$,
	where $d \neq n$. There are $l/(n/d)$ loops among the first $d$ vertices, so if $l_0 \neq 0$, there cannot be such a $P$.
	Hence assume that $l_0 = 0$. If $d \neq n/2$ and $e_0 \neq 0$, then for each edge $ij$
	in $P$ there must be edges $(i+d)(j+d),(i+2d)(j+2d),\dotsc, (i+n-d)(j+n-d)$ in $P$
	(where addition is taken modulo $n$). Hence the number of edges must be
	a multiple of $n/d$ which cannot be the case if $e_0 \neq 0$.

This exhausts all possibilities and thus the proof is complete.
\end{proof}

Recall that $\NCC(n+1, k)$ is the set of non-crossing $(1,2)$-configurations $P$ on $n$
vertices such that the number of loops plus proper edges of $P$ is equal to $k-1$.
In other words,
\[
	\NCC(n+1, k) = \bigcup_{i=0}^{k-1} \NCC(n+1,i,k-1-i).
\]
By applying \cref{lem:refinementOfqNarayana}, we obtain the following result.
\begin{corollary}\label{cor:thielRefinementNarayana}
For every $n, k\in \setN$ such that $0 \leq k \leq n+1$,
\[
	\left( \NCC(n+1,k), \langle \rot_n \rangle, \Nar[q]{n+1}{k} \right)
\]
exhibits the cyclic sieving phenomenon.
\end{corollary}
There is already a known instance of the cyclic sieving phenomenon with
the $q$-Narayana numbers as the polynomial, namely that of non-crossing partitions with a
fixed number of blocks and where the group action is
rotation \cite[Thm 7.2]{ReinerStantonWhite2004}.
Note, however, that in \cref{cor:thielRefinementNarayana} the cyclic
group has a different order than the one with non-crossing partitions.

\begin{remark}
	We cannot hope to find a refinement of the above CSP result involving the Kreweras numbers as in \cite{ReinerSommers2018}.
	For example, consider $n=4$ and $k=2$. There are two partitions of $n$ into $k$ parts, namely $(3,1)$ and $(2,2)$.
	There are $4$ non-crossing partitions with parts
	given by $(3,1)$ and $2$ non-crossing partitions with parts given by $(2,2)$.
	But $\NCC(4,2)$ has two orbits under rotation, both of size $3$.
\end{remark}

\section{Case \texorpdfstring{$s=n$}{s=n} and type \texorpdfstring{$B$}{B} Catalan numbers}\label{sec:s=n}

In this section, we prove several instances of the CSP, related to type $B$ Catalan numbers.
We first consider a $q$-Narayana refinement on non-crossing matchings.
In the subsequent subsection, we consider a cyclic descent refinement on binary words.
Finally, in the last subsection we prove a type $B$ analog of \cref{thm:thielRefinement}.

\subsection{Type \texorpdfstring{$B$}{B} Narayana CSP}

A \defin{type $B$ non-crossing partition} of size $n$ is a non-crossing
partition of $\{1,\dotsc,n,n+1,\dotsc, 2n \}$ which is preserved under a half-turn rotation.
These were first defined by Reiner in \cite{Reiner1997}.
We let this set be denoted $\mdefin{\NCP^B(n)}$ and
let $\rotB_n$ denote the action on $\NCP^B(n)$ by rotation of $\pi/n$.
Note that we only need to make a half-turn before arriving at the initial position.

\begin{proposition}
The triple
\[
 \left( \NCP^B(n), \langle \rotB_n \rangle, \qbinom{2n}{n} \right)
\]
is a CSP triple.
\end{proposition}
\begin{proof}
There are many ways to prove this. For example, $\NCP^B(n)$ can first be put in bijection
with type $B$ non-crossing matchings, which are non-crossing matchings on $4n$
vertices that are symmetric under a half-turn, by using \cref{bij:NCPtoNCM}.

We then consider the first $2n$ new vertices,
and for each vertex $u$, we record a $1$ if the edge $u \to v$
is oriented clockwise, and $0$ otherwise.
This is a binary word of length $2n$ with $n$ ones.
Furthermore, $\rotB_n$ of the non-crossing partition corresponds
to $\shift^2_{2n}$ on the binary word. The triple
$\left(\BW(2n,n), \langle \shift_{2n} \rangle, \qbinom{2n}{n}\right)$
exhibits the CSP (see \cite[Prop.~4.4]{ReinerStantonWhite2004}), so it is direct
from the definition of cyclic sieving that replacing $\shift_{2n}$ by
$\shift^2_{2n}$ also gives a CSP triple.
\end{proof}

We shall now consider Narayana refinements of type $B$
non-crossing partitions and non-crossing matchings.
First, we introduce the following polynomial:
\begin{equation}\label{eq:BNCP}
 \mdefin{\Pi_{n}(q;t)} \coloneqq \sum_{j=0}^{n} q^{j^2}\qbinom{n}{j} \left(
 t^{2j} q^{n-j} \qbinom{n-1}{j-1}
 +
 t^{2j+1} \qbinom{n-1}{j}\right).
\end{equation}
Note that by using the $q$-Pascal identity,
$\Pi_n(q;1) = \sum_{j} q^{j^2} \qbinom{n}{j}[2] = \qbinom{2n}{n}$, so the sum of the polynomials
\[
[t^0]\Pi_n(q;t), \quad
[t^1]\Pi_n(q;t), \quad
\dotsc,\quad
[t^{2n}]\Pi_n(q;t)
\]
refines the type $B$ $q$-Catalan numbers.
With the polynomial formulated, cyclic sieving is
easy to prove by following the proof of \cite[Thm 7.2]{ReinerStantonWhite2004}.
As a side note, the coefficients of the polynomial at $q=1$ are given by the OEIS entry \oeis{A088855}.

\begin{proposition}\label{prop:refinedTypeBNCP}
Let $n,k\geq 0$ be integers. Then
\[
 \left( \{ P \in \NCP^B(n): \blocks(P) = k \},
 \langle \rotB_n \rangle,
 [t^k] \Pi_{n}(q;t)
  \right),
\]
and
\begin{equation}\label{eq:TypeB-NCM-NarayanaBlock}
 \left( \{ P \in \NCP^B(n): 2k\leq \blocks(P) \leq 2k+1 \},
 \langle \rotB_n \rangle, q^{k^2} \qbinom{n}{k}[2]
  \right)
\end{equation}
exhibit the cyclic sieving phenomenon.

Moreover, for every $n\geq 1$ and $k$, $0 \leq k \leq n$,
\[
 \left( \{ M \in \NCM^B(n): \even(M) = k-1 \},
 \langle \rotB_n \rangle,
 [t^k] \Pi_{n}(q;t)
  \right)
\] and
\begin{equation}
 \left( \{ M \in \NCM^B(n): 2k-1\leq \even(M) \leq 2k \},
 \langle \rotB_n \rangle,
 q^{k^2} \qbinom{n}{k}[2]
 \right)
\end{equation}
exhibit the cyclic sieving phenomenon.
\end{proposition}
\begin{proof}
Everything is trivial unless $1 \leq k \leq n$, so we assume this holds. Using  \cref{bij:NCPtoNCM} the first two statements are equivalent to the
last two, so we only need to prove the former.
The number of half-turn symmetric non-crossing
partitions with $2k$ blocks is $\frac{k}{n}\binom{n}{k}^2$ and with $2k+1$ blocks $\frac{n-k}{n}\binom{n}{k}^2$.
This can be proven in different ways, but it suffices to refer to \cite[Lem.~4.4]{AthanasiadisReiner2004}.

Divide the numbers into $2d$ intervals $t\frac nd+1,\dotsc, (t+1)\frac nd$, $t\in\{0,\dotsc, 2d-1\}$. If a partition $P$ satisfies $\rotB_n^{n/d}(P)=P$,
 a block only contains numbers from two adjacent intervals or it is a central block with numbers from every interval. Let $r$ be the number of blocks
 that contain numbers from the intervals with $t=0$ and $t=1$, but no other. Then the total number
of blocks is $2dr$ or $2dr+1$, the latter if there is also a central block. In the proof of \cite[Thm 7.2]{ReinerStantonWhite2004}, they show that
the number of partitions $P \in \NCP^B(n)$ invariant under
$\rotB_n^{n/d}$ with $2dr$ and $2dr+1$ blocks are
\[
\frac{dr}{n}\binom{n/d}{r}^2 \text{ and }
\frac{n-dr}{n}\binom{n/d}{r}^2\!, \text{ respectively.}
\]
We now evaluate $[t^k] \Pi_{n}(q;t)$ at a primitive
$d^\thsup$ root of unity. The case $d=1$ is trivial.
If $k\neq 0,1 \pmod d$,  $d\ge 2$, then it is clearly zero.
For $k=2dr$, we get
\[
[t^{2dr}]\Pi_{n}(q;t) = q^{(dr)^2+n-dr}\qbinom{n}{dr} \qbinom{n-1}{dr-1}
\]
which by \cref{thm:q-Lucas} (the $q$-Lucas~theorem) becomes
$\binom{n/d}{r}\binom{n/d-1}{r-1}=\binom{n/d}{r}^2\frac{r}{n/d}$,
which is what we want. A similar calculation gives the case $k=2dr+1$.
The expression $q^{k^2} \qbinom{n}{k}[2]$ in \eqref{eq:TypeB-NCM-NarayanaBlock} is just the sum of the two
cases.
\end{proof}
A cyclic sieving result involving type $B$ Kreweras numbers (and thus type $B$ Catalan numbers)
was proven in \cite[Thm.~1.7]{ReinerSommers2018}.
The downside is that the Kreweras numbers in type $B$ are not indexed by usual partitions,
but partitions of $2n+1$, where each even part has even multiplicity.

\subsection{A second refinement of the type $B$ Catalan numbers}

Let $\BW(2n,n)$ be the set of binary words of length $2n$ with exactly $n$ ones.
Define a cyclic descent of a binary word $\sfb = \sfb_1 \sfb_2 \dotsb \sfb_{2n}$ as an
index $i$ such that $\sfb_i > \sfb_{i+1}$, where the indices are taken modulo $2n$.
The number of cyclic descents of $\sfb$ is denoted $\cdes(\sfb)$.
As an example, if $\sfb=\mathsf{0110010111}$, then $\cdes(\sfb)=3$.
For any two natural numbers $n$ and $k$, let $\mdefin{\BW^k(n)} \subset \BW(2n,n)$
consist of all $\sfb \in \BW(2n,n)$ such that $\cdes(\sfb)=k$.
Define
\begin{equation}\label{eq:maj polynomial bin words}
\mdefin{\qBW[q]{n}{k}} \coloneqq \sum_{\sfb \in \BW^k(n)} q^{\maj(\sfb)}.
\end{equation}
At $q=1$, this is \oeis{A335340} in the OEIS and two times \oeis{A103371}.
Note that we have $\qBW[q]{0}{0}=1$ and $\qBW[q]{n}{k}=0$ if $k>n$.
\begin{lemma}\label{lem:typeBSecondNarayana}
For all integers $1 \leq k \leq n$,
\begin{equation}\label{eq:typeBSecondNarayanaPol}
\qBW[q]{n}{k} = q^{k(k-1)}(1+q^n)\qbinom{n}{k}\qbinom{n-1}{k-1}.
\end{equation}
\end{lemma}
\begin{proof}
The set $\BW^k(n)$ is in bijection with a certain subset of $\PATHS(n)$
which we shall now describe. Call binary words of the form $\sfb=\mathsf{0} \sfb_2 \sfb_3 \dotsb \sfb_{2n-1}\mathsf{1}$ \defin{elevated} and call binary words that are not elevated \defin{non-elevated}
(so a binary word is elevated if $\cdes(\sfb)=\des(\sfb)+1$).
Elevated binary words in $\PATHS(n)$ are in
natural bijection with paths in $\PATHS(n-1)$ by letting the
elevated binary word $\mathsf{0} \sfb_2 \sfb_3 \dotsb \sfb_{2n-1} \mathsf{1}$
correspond to the binary word $\sfb_2 \sfb_3\dotsb \sfb_{2n-1}$.

It follows that a word in $\BW^k(n)$ corresponds either to a
non-elevated path in $\PATHS(n)$ with $k$ valleys or to an elevated
path in $\PATHS(n)$ with $k-1$ valleys.
Using this correspondence and \eqref{eq:B Narayana formula},
one gets that
\begin{equation}\label{eq:differenceOfPathByPeaks}
\qBW[q]{n}{k}= q^{k^2}\qbinom{n}{k}[2] - q^k \cdot q^{k^2}\qbinom{n-1}{k}[2]
			+ q^{k-1}\cdot q^{(k-1)^2}\qbinom{n-1}{k-1}[2].
\end{equation}
Here, the factors $q^k$ and $q^{k-1}$ appear since by translating
a binary word $\mathsf{c} \in \BW(2(n-1),n-1)$ with $k$ descents into its corresponding
elevated binary word $\mathsf{c}'$ in $\BW(2n,n)$, we have $\maj(P')-\maj(P)=k$
as each descent of $\mathsf{c}'$ contributes one more to $\maj$ than in $\mathsf{c}$.

It remains to show that the expression in \eqref{eq:differenceOfPathByPeaks}
coincides with the one in \eqref{eq:typeBSecondNarayanaPol}.
To do this, we rewrite
\begin{align*}
&q^{k^2}\qbinom{n}{k}[2] - q^k \cdot q^{k^2}\qbinom{n-1}{k}[2]  + q^{k-1}\cdot q^{(k-1)^2}\qbinom{n-1}{k-1}[2] \\
=\ &q^{k(k-1)} \qbinom{n}{k}\qbinom{n-1}{k-1}\left( q^k \frac{[n]_q}{[k]_q}-q^{2k} \frac{[n-k]_q^2}{[n]_q[k]_q}+\frac{[k]_q}{[n]_q}\right) \\
=\ &q^{k(k-1)} \qbinom{n}{k} \qbinom{n-1}{k-1} \left( \frac{q^k[n]_q^2-q^{2k}[n-k]_q^2 + [k]_q^2}{[n]_q[k]_q}\right).
\end{align*}
It is therefore sufficient to show that the expression inside
the parentheses is equal to $1+q^n$ or, equivalently, that the following equation holds:
\begin{equation*}\label{eq: q-integer2}
q^k[n]_q^2-q^{2k}[n-k]_q^2+[k]_q^2=(1+q^n)[n]_q[k]_q.
\end{equation*}
This equation can be derived from \eqref{eq: q-integer} by
adding $q^{k+n-1}[n]_q+q^{2k-1}[n-k]_q+q^{k-1}[k]_q=q^{k-1}(1+q^n)[n]_q$
to each side of the equation. This concludes the proof.
\end{proof}

The number of cyclic descents of a binary word is clearly invariant under
cyclic shifts of the word so one has a group action of $\rot_{2n}$ on $\BW^k(n)$.
The following proposition follows from \cite[Cor.~1.6]{AhlbachSwanson2018},
although they do not compute the closed-form expression of \cref{eq:maj polynomial bin words}.

\begin{proposition}\label{prop:typeBCdesCSP}
For all $n, k \in \setN$ such that $1 \leq k \leq n$, the triple
\[
\left( \BW^k(n), \shift_{2n}, \qBW[q]{n}{k} \right)
\]
exhibits the cyclic sieving phenomenon.
\end{proposition}

\subsection{Type \texorpdfstring{$B$}{B} non-crossing configurations with a twist}\label{sec:typeBNCC}

Recall that $\NCC(n+1)$ denotes the set of non-crossing $(1,2)$-configurations on $n$ vertices.
We shall now modify this family slightly.

\begin{definition}
Let $\mdefin{\NCC^B(n)}$ be the set  of non-crossing $(1,2)$-configurations on $n-1$ vertices,
with the extra option that one of the proper edges may be marked.
We let $\mdefin{\NCC^B(n,e,l)}\subset \NCC^B(n)$ be the subset
with exactly $e$ proper edges, and $l$ loops.
Finally, let $\NCC^B(n,k)$ be the subset of $\NCC^B(n)$ with $k$ edges and loops,
i.e.
\[
 \mdefin{\NCC^B(n,k)} \coloneqq \bigcup_{e+l = k} \NCC^B(n,e,l).
\]
\end{definition}

It follows directly from the definition that
$|\NCC^B(n,e,l)|=(e+1)|\NCC(n,e,l)|$ and it is not difficult to
sum over all possible $e,l$ to prove that $|\NCC^B(n+1,k)| = \NarB{n}{k} = \binom{n}{k}^2$.

\begin{theorem}\label{thm:typeB-NCC-Twist}
We let $\twist_{2n}$ act on $\NCC^B(n+1)$ as before (the marked edge is also rotated),
which gives an action of order $2n$.
Then
\begin{equation}
 \left( \NCC^B(n+1), \langle \twist^2_{2n} \rangle , \qbinom{2n}{n} \right)
\end{equation}
is a CSP triple.
\end{theorem}
\begin{proof}

We compute the number of fixed points under $(\twist^2_{2n})^d$ where we can without
loss of generality assume $d \mid n$.
Write $n = md$. By \cref{thm:q-Lucas},
\[
\qbinom{2n}{n} = \binom{2d}{d}
\]
at a primitive $m^\thsup$ root of unity.
The claim follows from \cref{thm:newCatalanCSP} except in the cases where a
marked edge can appear in a fixed point. Note that in the case $4 \mid n$ and
$2d = n/2$ or $3n/2$ a marked edge would have to split the configuration into
two non-crossing matchings on an odd number of vertices. Hence there cannot be
a marked edge in a fixed point in this case.

The only case left is $n \mid 2d$. First, $2d = 2n$ is trivial. Second,
if $2d = n$, no fixed point can have marked vertices, as is noted in the proof
of \cref{thm:newCatalanCSP}. Hence we only have non-crossing matchings on
$2d$ vertices with one edge possibly marked, the number of which is
$(d+1)\Cat{d} = \binom{2d}{d}$.
\end{proof}
It should be possible to prove \cref{thm:typeB-NCC-Twist} bijectively.

\begin{problem}
Find an equivariant bijection between $\NCC^B(n+1)$ and $\BW(2n, n)$ sending $\twist_{2n}^2$ to $\shift_{2n}^2$.
\end{problem}
Note that the triple in \cref{thm:typeB-NCC-Twist} exhibits the so-called
Lyndon-like cyclic sieving~\cite{AlexanderssonLinussonPotka2019},
which is not intuitively clear (as it is for $\BW(2n, n)$).


\begin{remark}
 \cref{thm:typeB-NCC-Twist} does not hold when only considering $\twist_{2n}$.
 For $n=2$, $\qbinom{2n}{n}$ evaluated at a primitive $4^\thsup$ root of unity gives $0$.
 However, there are $6$ elements in $\NCC^B(3)$, two of which are fixed under $\twist_{4}$;
 consider an edge between vertices $1$ and $2$, which may or may not be unmarked.
 Since there are no loops or isolated vertices, these two elements are fixed.
 Can one modify the $q$-analog of $\binom{2n}{n}$ so that it is compatible with $\twist_{2n}$?
\end{remark}

\begin{problem}
Is it possible to define a refinement $P(n,e,l;q)$ of $\qbinom{2n}{n}$ so that
\[
 \left(
 \bigcup_{l=0}^n \NCC^B(n+1,e,l), \langle \twist^2_{2n} \rangle ,
 \sum_{l\geq 0} P(n,e,l;q)\right)
 \] is a CSP triple?
\end{problem}

Unfortunately the polynomials
\[\mdefin{\qNCCB[q]{n}{e}{l}}\coloneqq
 q^{e^2 + n l} [e+1]_q \qbinom{n}{2e} \Cat[q]{e} \qbinom{n-2e}{l}\]
do not serve this purpose even though they do satisfy the identities (proof omitted)
 \begin{align*}
  \qNCCB[1]{n}{e}{l} &= |\NCC^B(n+1,e,l)|,\\
 \NarB[q]{n}{k} &= \sum_{e+l=k} \qNCCB[q]{n}{e}{l}.
\end{align*}


\subsection{Thiel's CSP for type \texorpdfstring{$B$}{B}}\label{sec:typeBThiel}

\begin{theorem}\label{thm:typeB-NCC-rot-CSP}
We let rotation $\rot_n$ act on $\NCC^B(n+1,e,l)$, and let
\[
\qNCC[q]{n}{e}{l} \coloneqq q^{e(e+1)+(n+1)l} \qbinom{n}{2e} \Cat[q]{e} \qbinom{n-2e}{l}.
\]
Then
\begin{equation}\label{eq:CSPTripleTypeBNCC}
 \left( \NCC^B(n+1,e,l), \langle \rot_n \rangle, (1+[e]_q) \qNCC[q]{n}{e}{l}
  \right)
\end{equation}
is a CSP triple.
\end{theorem}
\begin{proof}
We can split $\NCC^B(n+1,e,l)$ into two sets, the first
set $A$ being the case without a marked edge,
and the second set $B$ the case with a marked edge.
Then it suffices to prove that
\[
 \left(A,  \langle \rot_n \rangle, \qNCC[q]{n}{e}{l} \right)
 \quad
 \text{ and }
 \quad
 \left(B,  \langle \rot_n \rangle, [e]_q \qNCC[q]{n}{e}{l} \right)
\]
are CSP triples.
The first one is already proved in \cref{thm:thielRefinement}.

For the second, consider $\rot_n^d$, and without loss of generality write $n = kd$.
A single marked edge can only appear in a fixed point if $d = n$ or $d = n/2$.
The former is trivial. Now, rewrite the polynomial as
\[
q^{e(e+1)+(n+1)l} \qbinom{n}{2e} \qbinom{2e}{e-1} \qbinom{n-2e}{l},
\]
and apply \cref{thm:q-Lucas}. At a primitive $k^\thsup$ root of unity,
this evaluates to 0 unless $k \mid 2e$, $k \mid e-1$ and $k \mid l$.
The second implies $\gcd(k,e)=1$, so by the first $k \mid 2$. If $k = 2$,
that is $d = n/2$, we get that the number of fixed points should be
\[
\binom{\frac{n}{2}}{e} \binom{e}{\frac{e-1}{2}} \binom{\frac{n}{2}-e}{\frac{l}{2}}.
\]
This is indeed the case. The marked edge has to split the configuration into two
symmetric parts, and connects $i$ to $i + n/2$ for some $1 \le i \le n/2$.
The symmetric configurations are on $n/2$ - 1 vertices, and have $(e-1)/2$ edges
and $l/2$ marked vertices each. The number of fixed points is hence
\[
\frac{n}{2}\binom{\frac{n}{2} - 1}{e-1}\Cat{(e-1)/2}\binom{\frac{n}{2}-1-(e-1)}{\frac{l}{2}} = \binom{\frac{n}{2}}{e} \binom{e}{\frac{e-1}{2}} \binom{\frac{n}{2}-e}{\frac{l}{2}}.
\]

\end{proof}

By summing over the cases when $e+l=k$, we get the following corollary:
\begin{corollary}\label{cor:typeB-NCC-rot}
We have a $q$-analog of the type $B$ Narayana numbers,
which admits the CSP triple
 \begin{equation*}
 \left(
 \NCC^B(n+1,k), \langle \rot_n \rangle, U_{n,k}(q)
   \right),
\end{equation*}
where $U_{n,k}(q) = \sum_{e=0}^k q^{e(e+1)+(n+1)(k-e)} (1+[e]_q) \qbinom{n}{2e} \Cat[q]{e} \qbinom{n-2e}{k-e}$.
\end{corollary}
\begin{proof}
As in the discussion before \cref{thm:typeB-NCC-Twist} it is not difficult to prove that when $q=1$,
we do indeed obtain $\binom{n}{k}^2$, so this is a $q$-Narayana refinement.
\end{proof}
Now, summing over all $k$ gives cyclic sieving on $\NCC^B(n+1)$ under rotation.
We leave it as an open problem to find a nice expression for $\sum_k U_{n,k}(q)$.

%
%
%

\section{Two-column semistandard Young tableaux}\label{sec:branden}

The \defin{Schur polynomial} $\schurS_{\lambda}(x_1,\dotsc,x_n)$
is defined as the sum
\[
 \mdefin{\schurS_{\lambda}(x_1,\dotsc,x_n)} \coloneqq \sum_{T \in \SSYT(\lambda,n)}
 \prod_{j \in \lambda } x_{T(j)}
\]
where $\SSYT(\lambda,n)$ is the set of semi-standard Young
tableaux of shape $\lambda$ with maximal entry at most $n$.
The product is taken over all labels in $T$.

P. Br{\"a}nd{\'e}n gave the following interpretation of $q$-Narayana numbers.
\begin{theorem}[{See \cite[Thm.~6]{Branden2004}.}]\label{thm:branden}
For $0\leq k \leq n-1$,
\begin{equation}
\Nar[q]{n}{k+1} = \schurS_{2^{k}}(q,q^2,\dotsc,q^{n-1}).
\end{equation}
\end{theorem}

There is a type $B$ analog of \cref{thm:branden}.
\begin{theorem}
For $0\leq k \leq n$,
\begin{equation}
q^{k(k+1)} \qbinom{n}{k}\qbinom{n}{k}
= \schurS_{2^{k}1^k/1^k}(q,q^2,\dotsc,q^{n}).
\end{equation}
\end{theorem}
\begin{proof}
We first note that $\schurS_{2^{k}1^k/1^k} = (\schurS_{1^k})^2$.
To compute $\schurS_{1^k}$, we simply sum over all $k$-subsets of $[n]$.
This gives immediately that
\[
 \schurS_{1^k}(q,q^2,\dotsc,q^n) = q^{k(k+1)/2} \qbinom{n}{k},
\]
and the theorem above follows.
\end{proof}

It is then reasonable to interpret
\begin{equation}\label{eq:schurInterpolation}
 \schurS_{2^{k}1^s/1^s}(q,q^2,\dotsc,q^{n})
\end{equation}
for $0 \leq s \leq k$ as an interpolation between type $A$ ($s=0$)
and type $B$ ($s=k$) $q$-Narayana polynomials.
Note that this approach is different from what
is sought after in \cref{sec:typeABNarayanaQuest}.
The expression in \eqref{eq:schurInterpolation} can easily be computed by
the dual Jacobi--Trudi identity, see \cite{Macdonald1995}.
We find that \eqref{eq:schurInterpolation} is equal to
\[
q^{k(k+1)} \left(
\qbinom{n}{k}[2] -
q^{(s+1)^2}\qbinom{n}{k-s-1}\qbinom{n}{k+s+1}
\right).
\]

The first part of the theorem below follows from combining \cite[Thm.~1.4]{Rhoades2010} and \cref{thm:branden}.
\begin{theorem}\label{thm:ssytNarayanaCSP}
Assume $1\leq k<n$ and let $\kprom_{n-1}$ act on $\SSYT(2^k,n-1)$ via so-called $k$-promotion,
so that $\kprom_{n-1}$ has order $n-1$.
Then
\[
 \left(\SSYT(2^k,n-1), \langle \kprom_{n-1} \rangle,  \Nar[q]{n}{k+1} \right)
\]
is a CSP triple.
Moreover, there is a cyclic group $\langle\varphi\rangle$ of
order $n$ acting on $\SSYT(2^k,n-1)$ such that
\[
\left(\SSYT(2^k,n-1), \langle\varphi\rangle,  \Nar[q]{n}{k+1} \right)
\]
is a CSP triple.
\end{theorem}
\begin{proof}
We can define the action $\varphi$ as follows.
Given $T\in\SSYT(2^k,n-1)$, define
\[
Q=\{j:1\le j\le n-1, \mid\!\!\{i:T_{i,1}\ge j\}\!\!\mid=\mid\!\! \{i:T_{i,2}\ge j\}\!\!\mid\}
\]
and
let $P$ be the subset of $Q$ containing numbers not occurring as entries in $T$.
Further, define $\ell=\max P$ and $b=\max Q$. If $P=\emptyset$, let $\ell=0$.
Now $\varphi(T)$ is defined as follows.
If $\ell=n-1$, then we just add one to every entry in $T$ and are done.
If $\ell\neq n-1$, then we add one to every entry in $T$ and
\begin{itemize}
 \item in the first column:  remove $b+1$, add 1 in increasing order;
 \item in the second column: remove $n$, add $\ell+1$  in increasing order.
\end{itemize}
This can be seen to be an action with the desired properties by referring to \cref{bij:SSYTtoNCP}
from $\SSYT(2^k,n-1)$ to $\NCP(n,k+1)$.
Then rotation one step of the latter set corresponds to $\varphi$ where
$P$ are the other elements in the same block as $n$, $b$ is the
smallest element in the block of $n-1$ (if $n-1\notin P$), and $\ell$ is
the largest element in the block of $n$ other than $n$ itself.
Clearly, $b+1$ must be removed from the first column since it will
be the smallest element in the block of $n$ in $\varphi(T)$, and $\ell+1$ must be added to the second since it will
be the largest element in the block containing 1 in $\varphi(T)$.
\end{proof}

\begin{bijection}\label{bij:SSYTtoNCP}
Let $T \in \SSYT(2^k,n-1)$. Starting from $i = 1$, consider $x_i = T_{i,2}$.
Find the largest $y \in T_{\ast, 1}, y \leq x_i$, which is not in
\[
P_{i-1} \coloneqq \bigcup_{j \in [i-1]} p_j,
\]
and set $y_i = y$. Let $p_i = \{z \in \setN : y_i \le z \le x_i, z \not\in P_{i-1}\}$.
Repeat for $i < k + 1$. Finally, let $p_{k+1} = [n] \setminus P_{k}$.
Then the blocks $p_1, \dotsc, p_{k+1}$ form a non-crossing partition
in $\NCP(n,k+1)$ by construction. Note that exactly one element from
each of $T_{\ast, 1}$ and $T_{\ast, 2}$ is contained in $p_i$.
Note also that this is in fact the unique non-crossing partition
in $\NCP(n,k+1)$ having parts whose smallest and largest
elements are $y_i$ and $x_i$ respectively.

The inverse of the above bijection can be described as follows.
Let $p_1 | p_2 | \dotsb | p_{k+1} \in \NCP(n,k+1)$ and assume $n \in p_{k+1}$.
Let the first column of $T$ consist of the smallest elements from $p_i, 1 \leq i \leq k$,
in increasing order from top to bottom, and the second column of $T$ of
the largest elements from the same blocks, also in increasing order.
To see why $T \in \SSYT(2^k,n-1)$, it suffices to suppose that we on some
row $i$ have $T_{i,1} > T_{i,2}$, so the smallest element
in the block containing $T_{i,2}$ must be $T_{j,1}$ for some $j < i$.
Then the smallest element in the block containing $T_{j,2}$ has to
be $T_{k,1}$ for some $j < k < i$, the smallest element in the
block containing $T_{j+1,2}$ some $T_{k',1}$ for $j < k' < i$, and so on.
We match all elements $T_{j,2}$, $\dots$, $T_{i-1,2}$ to elements
in the first column between $T_{j,1}$ and $T_{i,1}$, but this yields a contradiction.
Note that $T$ is the unique element of $\SSYT(2^k,n-1)$ whose first column consists of the
smallest elements of $p_1$, $\dots$, $p_k$ and whose
second column consists of the largest elements of these parts.
Hence it is clear that this indeed is the inverse.
\end{bijection}

Below is an example of \cref{bij:SSYTtoNCP} and $\varphi$, when $n=8$, $k=4$.
\[
\ytableausetup{boxsize=1.2em}
T=\ytableaushort{12, 23, 34, 77}\quad
\longleftrightarrow\quad
\{\{1, 4\}, \{2\}, \{3\}, \{7\}, \{5,6,8\}\}, \quad
\varphi(T)=\ytableaushort{13, 24, 35, 47}
\]

\begin{remark}
The bijection in the proof of \cite[Thm.~5]{Branden2004} together with \cref{bij:NCPtoDyck}
provides a different bijection between $\SSYT(2^k,n-1)$ and $\NCP(n,k+1)$ where,
if $T \in \SSYT(2^k,n-1)$, then $T_{1,1}$, $T_{2,1} - T_{1,1}, \dots, T_{k,1} - T_{k-1,1}, n - T_{k,1}$
are the sizes of the blocks with $T_{1,2}, T_{2,2}, \dots, T_{k,2}$ as the largest elements.

\end{remark}

There is a type $B$ version of \Cref{thm:ssytNarayanaCSP}.
\begin{theorem}\label{thm:ssytBNarayanaCSP}
	Let $\kprom_n$ act on $\SSYT(2^k1^k/1^k,n)$ via $k$-promotion, then
	\[
	\left(\SSYT(2^k1^k/1^k,n), \langle \kprom_{n} \rangle,  q^{k(k+1)} \qbinom{n}{k}^2 \right)
	\]
	is a CSP triple.
\end{theorem}
\begin{proof}
	We describe a bijection from $\SSYT(2^k1^k/1^k,n)$ to $\BW(n,k) \times \BW(n,k)$.
	Let $T \in \SSYT(2^k1^k/1^k,n)$. The corresponding pair of binary words $(\sfb_1, \sfb_2)$
	is constructed as follows. Write $\sfb_1=\sfb_{11} \dotsc \sfb_{1n}$ and let $\sfb_{1i}=1$
	if $T$ has an $i$ in the left column, and otherwise, let $\sfb_{1i}=0$. In an analogous way,
	let $\sfb_2$ be determined by the entries in the right column of $T$. It is easy to see that
	this bijection is equivariant with the corresponding group action being $\rot_n$ on the pair of binary
	words. Here, $\rot_n$ acts by cyclically shifting both of the words one step. It follows from \cite[Thm.~1.1]{ReinerStantonWhite2004} that
	\[\left(\BW(n,k) \times \BW(n,k), \langle \rot_n \rangle, \qbinom{n}{k}^2 \right)\]
	exhibits the cyclic sieving phenomenon.
	It is not hard to see that the two different CSP-polynomials agree at $n^\thsup$ roots of unity.
	This completes the proof.
\end{proof}
It is natural to ask if the first part of \cref{thm:ssytNarayanaCSP} and \cref{thm:ssytBNarayanaCSP}
generalize to the intermediate skew shapes. We would then hope that $k$-promotion acting on $\SSYT(2^{k}1^s/1^s ,n)$, for $1\le s \le k-1$,
has order $n$. However, this is not the case, as for $n=4$, $k=2$, $s=1$, the tableaux
\[
\ytableaushort{{\none}3,14,2}, \quad
\ytableaushort{{\none}1,24,3}, \quad
\ytableaushort{{\none}1,13,4}
\]
form an orbit under $k$-promotion, but we want a group action of order $4$.

Perhaps some other group action gives a CSP triple with \eqref{eq:schurInterpolation} as
the CSP-polynomial. In a recent preprint, Y.-T.~Oh and E.~Park~\cite{OhPark2020x}
the authors show some closely related results, regarding cyclic sieving on SSYT.

\section{Triangulations of \texorpdfstring{$n$}{n}-gons with \texorpdfstring{$k$}{k}-ears}\label{sec:earRefinement}

We shall now consider type $A$ triangulations of an $n$-gon.
The main result in this section is a refinement of the CSP instance on
triangulations of $n$-gons which are counted by $\Cat{n-2}$, see \cite[Thm.~7.1]{ReinerStantonWhite2004}.
We also mention a cyclic sieving result on type $B$ triangulations.

\subsection{Refined CSP on triangulations by considering ears}

Let $\mdefin{\TRI(n)}$ denote the set of triangulations of a regular $n$-gon.
If the vertices $i, i+1, i+2 \pmod{n}$ are pairwise adjacent for $T \in \TRI(n)$,
we say they form an \defin{ear} of $T$.
Let $\mdefin{\TRI_{\ear}(n,k)}$ denote the set of triangulations of a regular $n$-gon
with exactly $k$ ears, and let $\Tri{n}{k}$ be the cardinality of this set.
It was shown by F.~Hurtado and M.~Noy~\cite[Thm.~1]{HurtadoNoy1996} that
\begin{equation}\label{eq:HurtadoNoy}
\Tri{n}{k} = \frac{n}{k}\binom{n-4}{2k-4}\Cat{k-2} \cdot 2^{n-2k} \quad \text{ whenever }\quad 2 \leq k \leq \frac{n}{2}.
\end{equation}
We now introduce the following $q$-analog of the expression in \eqref{eq:HurtadoNoy}.
For integers $n$ and $k$ satisfying $2 \leq k \leq \frac{n}{2}$, let
\begin{equation*}
\mdefin{\Tri[q]{n}{k}} \coloneqq
q^{k(k-2)} \frac{[n]_q}{[k]_q} \qbinom{n-4}{2k-4}\Cat[q]{k-2}
\left( \sum_{j=0}^{n-2k} q^{j(n-2)}\qbinom{n-2k}{j}\right).
\end{equation*}

At a first glance, one might hope that there is an easier expression for $\Tri[q]{n}{k}$. However, note that $\Tri[q]{n}{k}$ is not palindromic in general. As an example, consider $\Tri[q]{6}{2}=1+q^4+q^5+q^8$. This means that one cannot hope to find a formula for $\Tri[q]{n}{k}$ which only involves products of palindromic polynomials. In particular, it cannot be a product of $q$-binomial coefficients.

The choice of $\Tri[q]{n}{k}$ is motivated by the following theorem.
\begin{theorem}\label{thm:earTriRefinement}
For all integers $n \geq 4$, we have that
\begin{equation}
\sum_k \Tri[q]{n}{k} = \Cat[q]{n-2}.
\end{equation}
In other words, the polynomials $\Tri[q]{n}{k}$ refine the $q$-Catalan numbers.
\end{theorem}
\begin{proof}
We first recall some notation from $q$-hypergeometric series,
where we use \cite[Appendix I-II]{GasperRahman2004} as the main reference.
We set
$
\mdefin{(a;q)_n} \coloneqq (1-a)(1-aq)\dotsm (1-aq^{n-1})
$
so that
\[
 \qbinom{m}{r} = \frac{(q;q)_m}{(q;q)_r (q;q)_{m-r}}
\quad
\text{ and }
\quad
[m]_q = \frac{ (q;q)_m}{(1-q)(q;q)_{m-1}}.
\]
We have, \cite[I.7--I.26]{GasperRahman2004}, that for all $a$,
\begin{align*}
(a; q)_{n+k}     &= (a;q)_{k}(a q^k;q)_{n},\\
(a; q)_{n-k}     &=
\frac{(a;q)_{n}}{(q^{1-n}/a;q)_k}\left(-\frac{q}{a}\right)^{k} q^{\binom{k}{2}-nk}, \\
(a q^{-n};q)_{n} &= q^{-\binom{n}{2}} \left( -\frac{a}{q} \right)^n (q/a;q)_{n}, \\
(a q^{-n};q)_{k} &= q^{-\binom{n}{2}} \left( -\frac{a}{q} \right)^n
\frac{(q/a;q)_{n}(a;q)_k}{(q^{1-k}/a;q)_n}, \\
(q^{-n};q)_{k} &= (-1)^k
q^{\binom{k}{2}-nk} \frac{(q;q)_{n}}{(q;q)_{n-k}}, \\
(a q^{n}; q)_{n} &= \frac{(a;q)_{2n}}{(a;q)_{n}}.
\end{align*}
Moreover, we let
\[
 \mdefin{\pFq{2}{1}{a,b}{c}{q;z}} \coloneqq
 \sum_{n \geq 0} \frac{(a;q)_n (b;q)_n}{(c;q)_n (q;q)_n} z^n.
\]
The $q$-Chu--Vandermonde identity \cite[II.6, II.7]{GasperRahman2004},
can be stated in the following two ways:
\begin{equation}\label{eq:qChuVandermonde}
\pFq{2}{1}{a,q^{-n}}{c}{q;q} = \frac{(c/a;q)_n}{(c;q)_n} a^n
\quad\text{ and }\quad
\pFq{2}{1}{a,q^{-n}}{c}{q;cq^n/a} = \frac{(c/a;q)_n}{(c;q)_n}.
\end{equation}

We are now ready to prove \cref{thm:earTriRefinement},
which is equivalent to proving that for all $n \geq 4$,
\begin{align}\label{eq:mainIdentity}
\sum_{k \geq 2}
q^{k(k-2)}
\frac{[n][n-1]}{[k][k-1]}
\qbinom{n-4}{2k-4}
\qbinom{2k-4}{k-2}
\sum_{j=0}^{n-2k}
q^{j(n-2)} \qbinom{n-2k}{j}
=
\qbinom{2n-4}{n-2}.
\end{align}
After shifting the $k$-indices by $2$, and the $n$-indices by $4$,
and multiplying both sides with $(1+q)$, we must show that
\begin{equation}\label{eq:qSeries}
\sum_{k,j}
R(k,j)
=
\frac{(1+q)\qbinom{2n+4}{n+2}}{[n+3][n+4]},
\end{equation}
where
\begin{align*}
R(k,j) &=
\frac{(1+q)q^{k(k+2)+j(n+2)}}{[k+1][k+2]}
\qbinom{n}{2k}
\qbinom{2k}{k}
\qbinom{n-2k}{j} \\
&=
\frac{q^{k(k+2)}}{(q;q)_{k}(q^3;q)_{k}}
\frac{q^{j(n+2)}(q;q)_{n}}{(q;q)_{n-2k-j}(q;q)_{j}} \\
& =
(-1)^j
\frac{
q^{n(2k+j)-\binom{2k+j}{2}}
q^{k(k+2)}}{(q;q)_{k} (q^3;q)_{k}}
\frac{q^{j(n+2)}(q^{-n};q)_{2k+j}}{(q;q)_{j}}.
\end{align*}
The right hand side of \eqref{eq:qSeries} simplifies
to $\frac{ (q^5;q)_{2n} }{(q^5;q)_{n}(q^3;q)_{n}}$.
There is no issue with extending the summation index in \eqref{eq:qSeries}
so that $k$, $j$ ranges over all integers since $R(k,j)$ vanishes unless $0\leq j \leq n-2k$.
By shifting the indexing, so
that $r \coloneqq k$, $s \coloneqq k+j$,
it suffices to prove that
\begin{equation}
 \sum_{r,s} S(r,s) = \frac{ (q^5;q)_{2n} }{(q^5;q)_{n}(q^3;q)_{n}}
\end{equation}
where
\[
 S(r,s) \coloneqq R(r,s-r) =
 (-1)^{s+r}
\frac{
q^{2n s+2s-\binom{s+r}{2}}
q^{r^2}}{(q;q)_{r} (q^3;q)_{r}}
\frac{(q^{-n};q)_{s+r}}{(q;q)_{s-r}}.
\]
By using the identities
\[
 (q^{-n};q)_{s+r} = (q^{s-n};q)_r (q^{-n};q)_s  \text{ and }
 (q;q)_{s-r} = \frac{(q;q)_s}{(q^{-s};q)_r} (-1)^r q^{\binom{r}{2} -rs},
\]
we have
\begin{align*}
 S(r,s) &=
 (-1)^{s+r}
\frac{
q^{2n s+2s-\binom{s+r}{2}}
q^{r^2}}{(q;q)_{r} (q^3;q)_{r}}
(q^{s-n};q)_r (q^{-n};q)_s
\frac{ (q^{-s};q)_r (-1)^r q^{-\binom{r}{2} + rs} }{  (q;q)_s } \\
&=
\frac{
 (-q^{2n+2})^{s}
q^{-\binom{s}{2}}
 (q^{-n};q)_s (q^{s-n};q)_r  (q^{-s};q)_r
}{(q;q)_{r} (q^3;q)_{r} (q;q)_s } \cdot q^{r}.
\end{align*}
Thus, \eqref{eq:mainIdentity} is equivalent to
\[
 \sum_{r,s}
 \frac{
 (-q^{2n+2})^{s}
q^{-\binom{s}{2}}
 (q^{-n};q)_s (q^{s-n};q)_r  (q^{-s};q)_r
}{(q;q)_{r} (q^3;q)_{r} (q;q)_s } \cdot q^{r}
=
\frac{ (q^5;q)_{2n} }{(q^5;q)_{n}(q^3;q)_{n}}.
\]
But this follows from substituting
$a = q^2 q^n$ and $c=q^3$ in the following claim
and then expanding the $q$-hypergeometric series,
together with using the fact that
$(q^5;q)_{2n}  = (q^5;q)_{n}(q^n q^5 ;q)_{n}$.

\textbf{Claim:} For non-negative integers $n$,
we have the identity
\begin{align}\label{eq:qChuVanHelp}
\sum_{k \geq 0}
 \frac{ (-a q^{n})^{k}
q^{-\binom{k}{2}}
(q^{-n};q)_{k} }{  (q;q)_{k}  }
\pFq{2}{1}{\frac{cq^k}{aq},q^{-k}}{c}{q;q}
&=
\frac{ (ac ;q)_{n}  }{ (c;q)_{n} }.
\end{align}
Applying the first $q$-Chu-Vandermonde identity, the left-hand side becomes
\[
 \sum_{k \geq 0}
 \frac{ (-a q^{n})^{k}
q^{-\binom{k}{2}}
(q^{-n};q)_{k} }{  (q;q)_{k}  }
\frac{(aq/q^k;q)_k}{(c;q)_k}\left( \frac{cq^k}{aq} \right)^k.
\]
Now, using the identity
$
(aq/q^k;q)_k = q^{-\binom{k}{2}} (-a)^k (1/a;q)_k
$
we see that the left-hand side of \eqref{eq:qChuVanHelp} is equal to
\[
\sum_{k \geq 0}
\frac{ (-a q^{n})^{k}
q^{-\binom{k}{2}} (q^{-n};q)_{k} }{  (q;q)_{k} }
\frac{q^{-\binom{k}{2}} (-a)^k (1/a;q)_k}{(c;q)_k}\left( \frac{cq^k}{aq} \right)^k.
\]
Simplification gives
\[
\sum_{k \geq 0}
(a c q^{n})^{k}
\frac{
(q^{-n};q)_{k} (1/a;q)_k }{ (c;q)_k (q;q)_{k} }
=
\pFq{2}{1}{1/a,q^{-n}}{c}{q; a c q^{n} }.
\]
This is now a special case of the second $q$-Chu-Vandermonde
identity and we are done.
\end{proof}

A curious observation is that \cref{thm:earTriRefinement} refines
the $q$-Catalan numbers in the same spirit as the following $q$-analog of Touchard's
identity~\cite[Thm.~1]{Andrews2010}, which states that
\[
 \Cat[q]{n+1} = \sum_{r \geq 0}q^{2r^2+2r} \qbinom{n}{2r} \Cat[q]{r}\frac{(-q^{r+2};q)_{n-r}}{(-q;q)_{r}}.
\]

We let $\rot_n$ act on $\TRI(n)$ by rotating a triangulation one step clockwise.
As $\rot_n$ also preserves the set $\TRI_{\ear}(n,k)$,
we have a group action of $\langle \rot_n \rangle$ on $\TRI_{\ear}(n,k)$.
\begin{theorem}\label{thm:refinedTriCSP}
	For all integers $2 \leq k \leq \frac{n}{2}$, the triple
	\[
	\left( \TRI_{\ear}(n,k), \langle \rot_n \rangle, \Tri[q]{n}{k}\right)
	\]
	exhibits the cyclic sieving phenomenon.
\end{theorem}
\begin{proof}
Let $\xi$ be a primitive $n^\thsup$ root of unity and let $d$ be an integer.
Repeatedly using \cref{thm:q-Lucas} and \cref{lem:qFraction} yields
\[
\Tri[\xi^d]{n}{k} = \begin{cases}
\frac{n}{k}2^{n-2k}\binom{n-4}{2k-4}\Cat{k-2} & \text{if }d=n\\[1em]
\frac{n}{k}2^{n/2-k}\binom{n/2-2}{k-2}\binom{k-2}{k/2-1} & \text{if } \order(\xi^d)=2 \text{ and } 2 \mid k\\[1em]
\frac{n}{k}2^{n/3-2k/3}\binom{n/3-2}{2k/3-2}\binom{2k/3-2}{k/3-1} & \text{if } \order(\xi^d)=3 \text{ and } 3 \mid k\\[1em]
0 & \text{otherwise.}
\end{cases}
\]
We prove that these evaluations agree with the number of
fixed points in $\TRI_{\ear}(n,k)$ under $\rot_n^d$ on a case-by-case basis.

\noindent
\textbf{Case $d=n$}: Trivial.

\noindent
\textbf{Case $\order(\xi^d)=2$ and $2 \mid k$}:
Such a triangulation must have a diagonal (an edge from some $i$ to $i+d$)
that divides the triangulation into two halves.
The diagonal can be chosen in $n/2$ ways.
The triangulation is now determined uniquely by one of its halves.
To choose one half, we choose a triangulation of a $(n/2+1)$-gon with $k/2$ ears
whose sides do not coincide with the diagonal. Such a triangulation is either
(i) an element of $\TRI_{\ear}(n/2+1,k/2)$ which does not have an ear with an edge coinciding with
the diagonal or
(ii) an element of $\TRI_{\ear}(n/2+1,k/2+1)$ which has an ear coinciding with the diagonal.
Based on the rotational symmetry, we note that $2k/n$ of the elements in $\TRI_{\ear}(n,k)$
have an ear that has an edge adjacent to a given side.
Thus, the number of triangulations fixed by $\rot_n^{n/2}$ is equal to
\[
\frac{n}{2}\left( \frac{n/2+1-k}{n/2+1}\Tri{n/2+1}{k/2} + \frac{2(k/2+1)}{n/2+1}\Tri{n/2+1}{k/2+1}\right).
\]
Sage \cite{Sage} confirms that this symbolic expression agrees with the one given by
evaluating $\Tri[\xi^d]{n}{k}$.

\noindent
\textbf{Case $\order(\xi^d)=3$ and $3 \mid k$}:
Similar to the above case. Such a triangulation must have a
central triangle (a triangle with vertices $i$, $i+n/3$ and $i+2n/3$ $\pmod{n}$) that
divides the triangulation into three parts. The central triangle can be chosen in $n/3$ ways.
The triangulation is now determined uniquely by one of its three parts.
Choosing one part is done with a similar argument as above,
so the number of triangulations fixed by $\rot_n^{n/3}$ is equal to
\[
\frac{n}{3}\left(
\frac{n/3+1-2k/3}{n/3+1}\Tri{n/3+1}{k/3} + \frac{2(k/3+1)}{n/3+1}\Tri{n/3+1}{k/3+1}
\right).
\]
Sage confirms that this symbolic expression agrees with the one
given by evaluating $\Tri[\xi^d]{n}{k}$.

\noindent
\textbf{The remaining cases}:
If $\order(\xi^d) =2$ (or $3$), it is clear that any triangulation
fixed by $\rot_n^d$ must have a diagonal (or, respectively, a central triangle).
Thus each of the two halves (or, respectively, three parts)
must have the same number of ears and hence $2 \mid k$ (or $3 \mid k$).
If $\order(\xi^d) > 3$, then it is clear that there are no
fixed triangulations under the action of $\rot_n^d$.
This exhausts all possibilities and thus the proof is done.
\end{proof}

\begin{problem}
One might ask if there are further refinements.
However, note that $\frac{n}{k}\binom{n-2k}{j}\binom{n-4}{2k-4}\Cat{k-2}$, $0 \leq j \leq n-2k$, do not refine $\frac{n}{k}2^{n-2k}\binom{n-4}{2k-4}\Cat{k-2}$
since the former is not always an integer, for example, at $n = 5, k = 2, j = 1$.
\end{problem}

\begin{remark}
Unfortunately, there is no Narayana refinement of rotation acting on triangulations.
To see this, observe that $\Nar[q]{2}{2}=q^2$
but at a $4^\thsup$ root of unity $\xi$, we have $\Nar[\xi]{2}{2}=-1$.
\end{remark}

\subsection{Triangulations of type \texorpdfstring{$B$}{B}}

Let us define \defin{type $B$ triangulations} $\mdefin{\TRI^B(n)}$,
as the set of elements in $\TRI(2n)$ which are invariant under rotation by a half-turn.
In such a triangulation, there is always an edge through the center.
There are $n$ choices of such an edge, and then we need to choose a triangulation
on one half of the $2n$-gon. This gives $n\cdot \frac{1}{n}\binom{2(n-1)}{n-1} = \binom{2(n-1)}{n-1}$
such type $B$ triangulations. The following result is straightforward to prove but also follows
from \cite[Thm.~4.1]{EuFu2008}.

\begin{proposition}[{See \cite[Thm.~4.1]{EuFu2008}.}]\label{prop:typeBTriangulations}
The triple
\[
  \left( \TRI^B(n), \langle\rot_{n}\rangle, \qbinom{2n-2}{n-1} \right)
\]
exhibits the cyclic sieving phenomenon.
\end{proposition}

The polynomial $[n]_q \Tri[q]{n+1}{k}$ does not seem to give a refinement of
\cref{prop:typeBTriangulations}.

%
%

\section{Marked non-crossing matchings}\label{sec:holeyCatalan}

A \defin{marked non-crossing matching} is a non-crossing matching
where some of the regions 
have been marked.
Let $\mdefin{\NCM^{(r)}(n)}$ denote the set of
marked non-crossing matchings with exactly $r$ marked regions.
Since every non-crossing matching in $\NCM(n)$ has $n+1$
regions, it follows that  $|\NCM^{(r)}(n)|=\binom{n+1}{r}|\NCM(n)|$.

In particular, for $r=1$ we can think of our objects
as non-crossing matchings of vertices on the outer boundary of an annulus rather than on a disk.
\[
 \includegraphics[page=1,scale=0.75,valign=c]{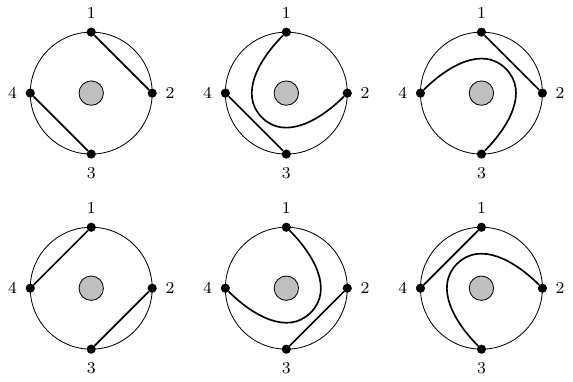} \\
\]
This model is reminiscent of the \emph{non-crossing permutations} considered in
\cite{Kim2013}, where points on the boundary of an annulus are matched in a non-crossing fashion,
but with some other technicalities imposed.

The following generalizes \cref{prop:NCMNarayanaCSP}.
\begin{theorem}
Let $1 \leq k \leq n$ and $0 \leq r \leq n+1$.
Then
\[
 \left(
 \{ M \in \NCM^{(r)}(n) : \even(M) = k \} ,
 \langle \rot_{n} \rangle,
 \Nar[q]{n}{k+1} \qbinom{n+1}{r}
 \right)
\]
is a CSP triple.
\end{theorem}
\begin{proof}
 Consider elements
of $\NCM(n)$ with $k$ even edges, fixed by $\rot_n^d$, where we may without
loss of generality assume $d \mid n$. As noted in \cref{sec:evenNCM}, $\even(M)$ is invariant under $\rot_n$.
Write $n = md$. By the $q$-Lucas theorem (\cref{thm:q-Lucas}),
\[
\qbinom{n+1}{r} = \begin{cases}
\binom{d}{r/m} & \text{ if $m \mid r$,} \\
\binom{d}{(r-1)/m} & \text{ if $m \mid r - 1$,} \\
0 & \text{ otherwise}
\end{cases}
\]
at a primitive $m^\thsup$ root of unity.

Divide the $2n$ vertices of a
non-crossing matching with $d$-fold rotational symmetry (remember that $\rot_n$ rotates the vertices two steps)  into $m$ segments of
length $2d$, say $[2d], [4d] \setminus [2d],$ and so on. The $d$ edges going from $[2d]$ to higher vertices (in $[4d]$ including $[2d]$) each have a unique region to their left.
This means that the number of marked regions $r$ must be $jm$, or $jm+1$ if there is a central region which is marked, see \cref{F:NCMNarayanaCSPex}.
We can thus get a fixed point of $\rot^d_n$ by choosing to mark $j = r/m$ (or, in the latter case, $(r-1)/m$) of the
$d$ regions associated to the edges from $[2d]$ to vertices with larger labels, which gives $\binom{d}{r/m} $ and $\binom{d}{(r-1)/m} $ respectively.
Now, we combine this with \cref{prop:NCMNarayanaCSP} and the theorem follows.

\end{proof}

\begin{figure}
\includegraphics[scale=1.2]{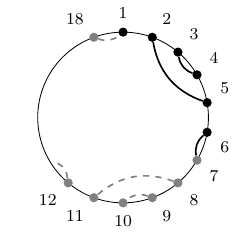}
\caption{Partitioning a non-crossing perfect matching of size $2n = 18$ with $d$-fold rotational symmetry, $d = 3$, into segments of length $2d$. Each of the three edges from $[2d]$ to vertices with bigger labels has a unique region to its left.}
\label{F:NCMNarayanaCSPex}
\end{figure}

\appendix

\section{Catalan objects}\label{sec:catalanObjectsAppendix}

\subsection{Type A objects}

Below is an overview of the Catalan objects we consider for $n=3$. Recall that $\Cat{3}=5$.

\begin{align*}
\SYT(n^2) &&&
\includegraphics[page=1,scale=0.5,valign=c]{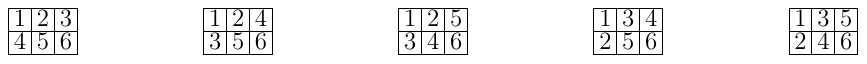} \\
\DYCK(n) &&&
\includegraphics[page=2,scale=0.5,valign=c]{appendixAFigures.pdf} \\
\NCM(n) &&&
\includegraphics[page=3,scale=0.5,valign=c]{appendixAFigures.pdf} \\
\NCC(n) &&&
\includegraphics[page=4,scale=0.5,valign=c]{appendixAFigures.pdf} \\
\NCP(n) &&&
\includegraphics[page=5,scale=0.5,valign=c]{appendixAFigures.pdf} \\
\SSYT(2^\ast,n-1) &&& \quad\;\;
\includegraphics[page=6,scale=0.5,valign=c]{appendixAFigures.pdf} \\
\TRI(n) &&&
\includegraphics[page=7,scale=0.5,valign=c]{appendixAFigures.pdf} \\
\end{align*}

\subsection{Type B objects}

Below is an overview of the considered type $B$ Catalan objects for $n=2$.
Recall that $\CatB{2}=6$.

\begin{align*}
\SYT(2n/n) &&&
\includegraphics[page=1,scale=0.5,valign=c]{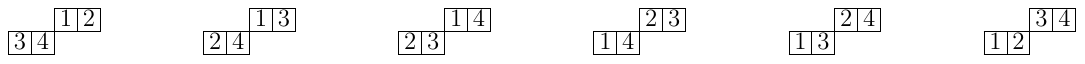} \\
\PATHS(n) &&&
\includegraphics[page=2,scale=0.5,valign=c]{appendixBFigures.pdf} \\
\NCM^B(n) &&&
\includegraphics[page=3,scale=0.5,valign=c]{appendixBFigures.pdf} \\
\NCC^B(n) &&&
\includegraphics[page=4,scale=0.5,valign=c]{appendixBFigures.pdf} \\
\NCP^B(n) &&&
\includegraphics[page=5,scale=0.5,valign=c]{appendixBFigures.pdf} \\
\TRI^B(n) &&&
\includegraphics[page=6,scale=0.5,valign=c]{appendixBFigures.pdf} \\
\end{align*}

\section{Bijections}\label{sec:bijections}

Here we recall several bijections on Catalan objects which have appeared earlier in the literature.
We have tried to find the earliest reference for each.

\subsection{NCM and binary words}\label{sec:NCMandBW}
Suppose $xy$ is an edge in a non-crossing perfect matching, with $x<y$.
We say that $x$ is the \defin{starting vertex} and $y$ is the \defin{end vertex}.
Further, denote $\NCM_{\sh}(n,k)$
the subset of all $N \in \NCM(n)$ such that $\sml(N)=k$.

We now describe a well-known bijection $\mdefin{\NCMtoDYCK}$ from $\NCM(n)$ to $\DYCK(n)$.
\begin{bijection}\label{bij:NCMtoDyck}
	Take $M\in \NCM(n)$ and construct the Dyck path
	$\NCMtoDYCK(M)=\sfb_1\sfb_2\dotsb \sfb_{2n}$ as follows.
	For vertices $i=1,2,\dotsc, 2n$ in $M$, let $\sfb_i=0$ if $i$ is a starting
	vertex and let $\sfb_i=1$ if $i$ is an end vertex.
	It is not hard to see that this procedure
	ensures that the resulting binary word is a Dyck path.
\end{bijection}

Let $d$ be a natural number such that $d\mid n$.
If a matching $M \in \NCM(n)$ has $d$-fold rotational symmetry,
it is sufficient to understand how the vertices $1,2,\dotsc, d$ are
matched up. Here, we restrict ourselves to the case when $M$ does not have a
diagonal---for the case when $M$ has a diagonal, see the
third case in the proof of \cref{thm:PMrefinedCSP}.
In this case, there is a bijection $\BWtoNCM$ between $\BW(2d,d)$ and
rotationally symmetric non-crossing perfect matchings.

\begin{bijection}\label{bij:BwToNCM}
Let $\sfc=\sfc_1 \sfc_2 \dotsb \sfc_{2d} \in \BW(2d,d)$.
We show how to construct the corresponding $\BWtoNCM(\sfc) \in \NCM(n)$.
Think of the vertices $1,2, \dotsc, 2d$ being arranged on a line.
For all vertices $i=1,2, \dotsc 2d$, we let
vertex $i$ be a ``left'' vertex if $\sfc_i=0$ and a ``right'' vertex if $\sfc_i=1$.
Then we match every left vertex with the closest available
right vertex to its right without creating a crossing of edges.
There is a unique way of doing this so that no
left vertex remains unpaired between a matched pair of vertices.
However, there might be some left vertices which are not matched because
there are not sufficiently many right vertices to their right.
There must also be equally many unpaired right vertices
because there are not enough left vertices to their left.
Since this will be the same in every interval of vertices $[2dk+1, 2d(k+1)]$
there is a unique way to pair the remaining left vertices with the remaining
right vertices of the interval to the right and vice versa.
\end{bijection}

We can prove a stronger statement about rotationally symmetric $\NCM$s.
We study $\BWtoNCM$ restricted to $\NCM_{\sh}(n,k)$. Once again, such a matching is
completely determined by how the vertices $1,2,\dotsc, 2d$ are paired up.
Among these $2d$ vertices, there are two possible cases to consider.

\textbf{Case 1:}
Exactly $dk/n$ short edges where either vertex $1$ is a left vertex or vertex $2d$ is a right vertex.

\textbf{Case 2:}
Exactly $dk/n-1$ short edges where vertex $1$ is a right vertex and vertex $2d$ is a left vertex.

If one applies $\BWtoNCM^{-1}$ (that is, left vertex corresponds
to $\mathsf{0}$ and right vertex corresponds to $\mathsf{1}$) to the
non-crossing perfect matchings in the two above cases,
one gets the image $\BW^{dk/n}(d)$.

\subsection{NCM and SYT}

We recall the definition of this bijection.
\begin{bijection}\label{bij:SYTtoNCM}
Let $T \in \SYT(n^2)$ and construct $\SYTtoNCM(T)$ as follows.
For $i\in \{1,2,\dotsc, 2n\}$, let vertex $i$ be a starting vertex
if $i$ is in the first row and an end vertex otherwise.
It is not hard to show that determining the starting
and ending vertices uniquely determines a non-crossing matching.
\end{bijection}

\subsection{NCP and NCM}

There is a bijection between non-crossing partitions and non-crossing matchings,
$\NCPtoNCM:\NCP(n) \to \NCM(n)$,
that directly restricts to a bijection between $\NCP^{B}$ and $\NCM^B$.
This bijection has another nice property;
it is equivariant with regards to the Kreweras complement on the non-crossing perfect
matchings and rotation on the non-crossing perfect matchings \cite[Thm.~1]{Heitsch2007}.

\begin{bijection}\label{bij:NCPtoNCM}
Consider $\pi \in \NCP(n)$, and for every vertex $j \in \{1,\dotsc,n\}$,
insert a new vertex $(2j-1)'$ immediately after $j$, and $(2j-2)'$ immediately before $j$,
where we insert $(2n)'$ immediately before $1$.
There we match $(2j)'$ to the closest point $(2k-1)'$, $j<k$, such that the edge
between those vertices does not cross any of the blocks in $\pi$.
If no such $k$ exists, we match to the smallest number $k, 0<k\le j$. Since $j=k$
is always possible the map is well-defined. This gives a
perfect matching $\sigma$ on $1',2',\dotsc,(2n)'$ which is non-crossing,
and in addition, does not cross any of the blocks of $\pi$. We can get the inverse
of the map $\NCPtoNCM$ by putting back the vertices $1,2,\dotsc,n$ and letting all
vertices that can have an edge between them without crossing any edge of
the perfect matching $\sigma$ belong to the same block.
\end{bijection}

The following is an example of this bijection.
\begin{equation}\label{eq:NCPtoNCMexample}
\NCPtoNCM:\quad
\matching{6}{3/3,6/6,1/2,2/4,4/5,5/1}{}
\qquad
\to
\qquad
\matching{12}{1/2,3/6,4/5,7/8,9/12,10/11}{}
\end{equation}

One can also illustrate the bijection as follows. For each singleton block,
add an edge between two copies of the vertex. For each block of size two, split the edge into
two non-crossing edges and hence each vertex into two copies. Finally, for each block with
$m \geq 3$ elements, push apart the edges at the vertices so that we have $m$ non-crossing edges
on $2m$ vertices. See \cref{fig:ncptoncm}.

\begin{figure}
\includegraphics[scale=1.5]{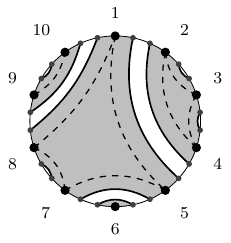}
\caption{A non-crossing partition given by the dashed edges and the corresponding non-crossing matching given by the solid edges.}
\label{fig:ncptoncm}
\end{figure}

Note that, by definition, every block except the one whose minimum element is 1
corresponds to an even edge under the bijection. If $j > 1$ is the smallest
element of a block and $k$ is the largest, then $(2j-2)'$ is matched to $(2k-1)'$.
The remaining even vertices are matched to smaller odd vertices. Hence \cref{bij:NCPtoNCM}
in fact gives a bijection between the sets
$\NCP(n,\ell)$ and $\NCM(n,\ell-1)$.

\subsection{NCP and Dyck paths}

We briefly describe a bijection between non-crossing partitions
and Dyck paths, $\mdefin{\NCPtoDYCK}$,
with the property that the number of parts is sent to the
number of peaks. This bijection is often attributed to P.~Biane~\cite{Biane1997}.

\begin{bijection}\label{bij:NCPtoDyck}
Let $\pi = B_1|B_2|\dotsb|B_k$ be a non-crossing partition of size $n$,
where the blocks are ordered increasingly according to maximal member.
By convention, we let $B_0$ be a block of size $0$,
where the maximal member is also $0$. We construct a Dyck path $D \in \DYCK(n)$
from $\pi$ as follows.
For each $j=1,\dotsc,k$, we have a sequence of
$\max(B_j)-\max(B_{j-1})$ north steps, immediately followed by $|B_j|$ east steps.
\end{bijection}

\bibliographystyle{alphaurl}
\bibliography{bibliography}

\end{document}